\documentclass[a4paper,11pt,fleqn]{article}
\usepackage[official]{eurosym}
\usepackage{amsmath}
\usepackage{amssymb}
\usepackage{theorem}
\usepackage{mathrsfs} %pour mathscr
\usepackage[usenames,dvipsnames]{pstricks}
\usepackage{euscript}
\usepackage{graphicx}
\usepackage{charter}
\PassOptionsToPackage{normalem}{ulem}
\usepackage{ulem}

\renewcommand{\leq}{\ensuremath{\leqslant}}
\renewcommand{\geq}{\ensuremath{\geqslant}}
\renewcommand{\le}{\ensuremath{\leqslant}}

\newcommand{\minimize}[2]{\ensuremath{\underset{\substack{{#1}}}%
{\text{\rm minimize}}\;\;#2 }}
\newcommand{\EC}[2]{{\mathsf E}(#1\! \mid\! #2)} 
\newcommand{\EEC}[2]{{\mathsf E}\bigg(#1\! \:\Big|\: #2\bigg)}

\newcommand{\scal}[2]{{\left\langle{{#1}\mid{#2}}\right\rangle}}

\newcommand{\menge}[2]{\big\{{#1}~\big |~{#2}\big\}} 
\newcommand{\Menge}[2]{\left\{{#1}~\Big|~{#2}\right\}} 
\newcommand{\HHH}{{\ensuremath{\boldsymbol{\mathsf H}}}}
\newcommand{\GGG}{{\ensuremath{\boldsymbol{\mathsf G}}}}
\newcommand{\KKK}{{\ensuremath{\boldsymbol{\mathsf K}}}}
\newcommand{\HH}{\ensuremath{{\mathsf H}}}
\newcommand{\GG}{\ensuremath{{\mathsf G}}}
\newcommand{\FF}{\ensuremath{{\EuScript F}}}
\newcommand{\XX}{\ensuremath{\EuScript{X}}}
\newcommand{\YY}{\ensuremath{\EuScript{Y}}}
\newcommand{\WC}{\ensuremath{{\mathfrak W}}}
\newcommand{\SC}{\ensuremath{{\mathfrak S}}}
\newcommand{\XXX}{\ensuremath{\boldsymbol{\EuScript{X}}}}
\newcommand{\EEE}{\ensuremath{\boldsymbol{\EuScript{E}}}}
\newcommand{\Sum}{\ensuremath{\displaystyle\sum}}
\newcommand{\emp}{\ensuremath{{\varnothing}}}

\newcommand{\Id}{\ensuremath{\text{\rm Id}}\,}
\newcommand{\ID}{\ensuremath{\text{\bfseries Id}}\,}
\newcommand{\cart}{\ensuremath{\raisebox{-0.5mm}{\mbox{\LARGE{$\times$}}}}}

\newcommand{\RR}{\ensuremath{\mathbb{R}}}
\newcommand{\RP}{\ensuremath{\left[0,+\infty\right[}}

\newcommand{\RPP}{\ensuremath{\left]0,+\infty\right[}}

\newcommand{\RPX}{\ensuremath{\left[0,+\infty\right]}}
\newcommand{\RX}{\ensuremath{\left]-\infty,+\infty\right]}}

\newcommand{\PP}{\ensuremath{\mathsf P}}
\newcommand{\as}{\ensuremath{\text{\rm $\PP$-a.s.}}}

\newcommand{\NN}{\ensuremath{\mathbb N}}

\newcommand{\weakly}{\ensuremath{\:\rightharpoonup\:}}
\newcommand{\exi}{\ensuremath{\exists\,}}

\newcommand{\zer}{\ensuremath{\text{\rm zer}\,}}
\newcommand{\pinf}{\ensuremath{{+\infty}}}

\newcommand{\dom}{\ensuremath{\text{\rm dom}\,}}
\newcommand{\prox}{\ensuremath{\text{\rm prox}}}
\newcommand{\Fix}{\ensuremath{\text{\rm Fix}\,}}

\newcommand{\gra}{\ensuremath{\text{\rm gra}\,}}

\newcommand{\zeroun}{\ensuremath{\left]0,1\right[}}   
\newcommand{\rzeroun}{\ensuremath{\left]0,1\right]}}   
\newtheorem{theorem}{Theorem}[section]
\newtheorem{lemma}[theorem]{Lemma}
\newtheorem{corollary}[theorem]{Corollary}
\newtheorem{proposition}[theorem]{Proposition}
\theoremstyle{plain}{\theorembodyfont{\rmfamily}%
\newtheorem{notation}[theorem]{Notation}}
\theoremstyle{plain}{\theorembodyfont{\rmfamily}%
}
\theoremstyle{plain}{\theorembodyfont{\rmfamily}%
}
\theoremstyle{plain}{\theorembodyfont{\rmfamily}%
}
\theoremstyle{plain}{\theorembodyfont{\rmfamily}%
}
\theoremstyle{plain}{\theorembodyfont{\rmfamily}%
\newtheorem{remark}[theorem]{Remark}}
\theoremstyle{plain}{\theorembodyfont{\rmfamily}%
\newtheorem{definition}[theorem]{Definition}}
\theoremstyle{plain}{\theorembodyfont{\rmfamily}%
}

\numberwithin{equation}{section}
\begin{document}
\title{\sffamily Stochastic Quasi-Fej\'er Block-Coordinate Fixed
Point Iterations\\ with Random Sweeping\thanks{Contact author: 
P. L. Combettes, {\ttfamily{plc@ljll.math.upmc.fr}}, phone: 
+33 1 4427 6319, fax: +33 1 4427 7200. 
This work was supported by the 
CNRS MASTODONS project under grant 2013MesureHD.}}
\author{Patrick L. Combettes$^{1}$ and 
Jean-Christophe Pesquet$^2$
\\[5mm]
\small
\small $\!^1$Sorbonne Universit\'es -- UPMC Univ. Paris 06\\
\small UMR 7598, Laboratoire Jacques-Louis Lions\\
\small F-75005 Paris, France\\
\small \ttfamily{plc@ljll.math.upmc.fr}
\\[5mm]
\small
\small $\!^2$Universit\'e Paris-Est\\
\small Laboratoire d'Informatique Gaspard Monge -- CNRS UMR 8049\\
\small F-77454, Marne la Vall\'ee Cedex 2, France\\
\small \ttfamily{jean-christophe.pesquet@univ-paris-est.fr}
}
\date{~}
\maketitle
\thispagestyle{empty}
\setcounter{page}{0}

\vskip 8mm

\begin{abstract}
This work proposes block-coordinate fixed point algorithms with
applications to nonlinear analysis and optimization in Hilbert 
spaces. The asymptotic analysis relies on a notion of 
stochastic quasi-Fej\'er monotonicity, which is thoroughly
investigated. The iterative methods under consideration feature 
random sweeping rules to select arbitrarily the blocks of 
variables that are 
activated over the course of the iterations and they allow for 
stochastic errors in the evaluation of the operators. Algorithms 
using quasinonexpansive operators or compositions of averaged 
nonexpansive operators are constructed, and weak and strong 
convergence results are established for the sequences they 
generate. As a by-product, novel block-coordinate operator 
splitting methods are obtained for solving structured monotone 
inclusion and convex minimization problems. 
In particular, the proposed framework leads to random 
block-coordinate versions of the Douglas-Rachford and 
forward-backward algorithms and of some of their variants.
In the standard case of $m=1$ block, our results remain new as they
incorporate stochastic perturbations.
\end{abstract}


\newpage

\section{Introduction}
The main advantage of block-coordinate algorithms is to result in  
implementations with reduced complexity and memory requirements 
per iteration. These benefits have long been recognized
\cite{Ausl71,Ceaj71,Orte70} and have become increasingly 
important in very large-scale problems. In addition, 
block-coordinate strategies may lead to faster \cite{Chou14} or 
distributed \cite{Iutz13} implementations. In this paper, we 
propose a block-coordinate fixed point algorithmic framework to 
solve a variety of problems in Hilbertian nonlinear numerical 
analysis and optimization. Algorithmic fixed point theory in 
Hilbert spaces provides a unifying and powerful framework for the 
analysis and the construction of a wide array of solution methods 
in such problems \cite{Baus96,Livre1,Cegi12,Opti04,Zeid86}. 
Although several block-coordinate algorithms exist for solving
specific optimization problems in Euclidean spaces, a framework 
for dealing with general fixed point methods in Hilbert spaces 
and which guarantees the convergence of the iterates does not seem 
to exist at present. In the proposed constructs, a random sweeping 
strategy is employed for selecting the blocks of coordinates which 
are activated over the iterations. The sweeping rule allows for
an arbitrary sampling of the indices of the coordinates.
Furthermore, the algorithms tolerate stochastic errors in the 
implementation of the operators. 
This paper provides the first general stochastic block-coordinate 
fixed point framework with guaranteed convergence of the iterates. 
It generates a wide range of new algorithms, which will be
illustrated by numerical experiments elsewhere.

A main ingredient for proving the convergence of many fixed point
algorithms is the fundamental concept of (quasi-)Fej\'er monotonicity
\cite{Else01,Ency09,Erem09,Raik67}. In Section~\ref{sec:2}, 
refining the seminal work of \cite{Ermo69,Ermo71,Ermo68}, we 
revisit this concept from a stochastic standpoint. By 
exploiting properties of almost super-martingales \cite{Robb71}, we 
establish novel almost sure convergence results for an 
abstract stochastic iteration scheme. In Section~\ref{sec:3}, 
this scheme is applied to the design of block-coordinate algorithms 
for relaxed iterations of quasinonexpansive operators. A simple
instance of such iterations is the Krasnosel'ski\u\i--Mann 
method, which has found numerous applications \cite{Livre1,Byrn14}. 
In Section~\ref{sec:4}, we design block-coordinate algorithms 
involving compositions of averaged nonexpansive operators.
The results are used in Section~\ref{sec:5} to construct
block-coordinate algorithms for structured monotone inclusion and 
convex minimization problems. Splitting algorithms have recently 
becomes tools of choice in signal processing and machine learning; 
see, e.g., 
\cite{Byrn14,Banf11,Smms05,Devi11,Ragu13,Wrig11}. Providing
versatile block-coordinate versions of these algorithms is expected 
to benefit these emerging areas, as well as more traditional 
fields of applications of splitting methods, e.g., \cite{Glow89}.
One of the offsprings of our work is an original
block-coordinate primal-dual algorithm which can be employed to
solve a large class of variational problems.

\section{Stochastic quasi-Fej\'er monotonicity}
\label{sec:2}

Fej\'er monotonicity has been exploited in various 
areas of nonlinear analysis and optimization to unify the 
convergence proofs of deterministic algorithms; see, e.g., 
\cite{Livre1,Ency09,Erem09,Raik67}. In the late 1960s, this 
notion was revisited in a stochastic setting in Euclidean spaces 
\cite{Ermo69,Ermo71,Ermo68}. In this section, we investigate a
notion of stochastic quasi-Fej\'er monotone sequence in Hilbert 
spaces and apply the results to a general stochastic iterative 
method. Throughout the paper, the following notation will be used.  

\begin{notation}
\label{n:2013}
$\HH$ is a separable real Hilbert space with 
scalar product $\scal{\cdot}{\cdot}$, associated norm 
$\|\cdot\|$, and Borel $\sigma$-algebra $\mathcal{B}$. 
$\Id$ denotes the identity operator on $\HH$ and
$\weakly$ and $\to$ denote, respectively, weak and strong 
convergence in $\HH$. The sets of strong and weak sequential 
cluster points of a sequence $(\mathsf{x}_n)_{n\in\NN}$ in $\HH$ 
are denoted by $\SC(\mathsf{x}_n)_{n\in\NN}$ and 
$\WC(\mathsf{x}_n)_{n\in\NN}$, respectively. 
The underlying probability space is 
$(\Omega,\FF,\PP)$. A $\HH$-valued random variable is a measurable 
map $x\colon(\Omega,\FF)\to(\HH,\mathcal{B})$. The 
$\sigma$-algebra generated by a family $\Phi$ of 
random variables is denoted by $\sigma(\Phi)$.
Let $\mathscr{F}=(\FF_n)_{n\in\NN}$ be a sequence of 
sub-sigma algebras of $\FF$ such that $(\forall n\in\NN)$ 
$\FF_n\subset\FF_{n+1}$.
We denote by $\ell_+(\mathscr{F})$ the set of sequences 
of $\RP$-valued random variables $(\xi_n)_{n\in\NN}$ such that,
for every $n\in\NN$, $\xi_n$ is $\FF_n$-measurable. We set
\begin{equation}
\label{eHj7+1-13}
(\forall p\in\RPP)\quad\ell_+^p(\mathscr{F})=
\Menge{(\xi_n)_{n\in\NN}\in\ell_+(\mathscr{F})}
{\sum_{n\in\NN}\xi_n^p<\pinf\;\as}
\end{equation}
and
\begin{equation}
\label{eHj7+1-12}
\ell_+^\infty({\mathscr{F}})=
\Menge{(\xi_n)_{n\in\NN}\in\ell_+(\mathscr{F})}
{\sup_{n\in\NN}\xi_n<\pinf\; \as}.
\end{equation}
Given a sequence $(x_n)_{n\in\NN}$ of $\HH$-valued 
random variables, we define
\begin{equation}
\label{eHj7+1-14}
{\mathscr{X}}=(\XX_n)_{n\in\NN},
\quad\text{where}\quad (\forall n\in\NN)\quad
\XX_n=\sigma(x_0,\ldots,x_n). 
\end{equation}
Equalities and inequalities involving random variables will always
be understood to hold $\PP$-almost surely, even if the expression 
``$\as$'' is not explicitly written.
For background on probability in Hilbert spaces, see 
\cite{Fort95,Ledo91}.
\end{notation}

\begin{lemma}{\rm \cite[Theorem~1]{Robb71}}
\label{l:1971}
Let $\mathscr{F}=(\FF_n)_{n\in\NN}$ be a sequence of sub-sigma 
algebras of $\FF$ such that 
$(\forall n\in\NN)$ $\FF_n\subset\FF_{n+1}$. Let
$(\alpha_n)_{n\in\NN}\in\ell_+({\mathscr{F}})$, 
$(\vartheta_n)_{n\in\NN}\in\ell_+({\mathscr{F}})$, 
$(\eta_n)_{n\in\NN}\in\ell_+^1({\mathscr{F}})$, and 
$(\chi_n)_{n\in\NN}\in\ell_+^1({\mathscr{F}})$ be such 
that
\begin{equation}
\label{eHj7+0-26}
(\forall n\in\NN)\quad
\EC{\alpha_{n+1}}{\FF_n}+\vartheta_n\leq(1+\chi_n)\alpha_n
+\eta_n\quad\as
\end{equation}
Then $(\vartheta_n)_{n\in\NN}\in\ell_+^1({\mathscr{F}})$ and
$(\alpha_n)_{n\in\NN}$ converges $\as$ to a $\RP$-valued random
variable.
\end{lemma}

\begin{proposition}
\label{p:1}
Let $\mathsf{F}$ be a nonempty closed subset of $\HH$, let 
$\phi\colon\RP\to\RP$ be a strictly increasing 
function such that $\lim_{t\to\pinf}\phi(t)=\pinf$, and let 
$(x_n)_{n\in\NN}$ be a sequence of $\HH$-valued random variables. 
Suppose that, for every $\mathsf{z}\in\mathsf{F}$, there exist
$(\chi_n(\mathsf{z}))_{n\in\NN}\in\ell_+^1({\mathscr{X}})$,
$(\vartheta_n(\mathsf{z}))_{n\in\NN}\in\ell_+({\mathscr{X}})$,
and $(\eta_n(\mathsf{z}))_{n\in\NN}\in\ell_+^1({\mathscr{X}})$
such that the following is satisfied $\as$:
\begin{equation}
\label{e:sqf1}
(\forall n\in\NN)\quad
\EC{\phi(\|x_{n+1}-\mathsf{z}\|)}{\XX_n}+\vartheta_n(\mathsf{z}) 
\leq (1+\chi_n(\mathsf{z}))\phi(\|x_n-\mathsf{z}\|)+
\eta_n(\mathsf{z}).
\end{equation}
Then the following hold:
\begin{enumerate}
\item
\label{p:1i}
$(\forall\mathsf{z}\in\mathsf{F})$ 
$\big[\:\sum_{n\in\NN}\vartheta_n(\mathsf{z})<\pinf\:\as\big]$
\item
\label{p:1ii}
$(x_n)_{n\in\NN}$ is bounded $\as$
\item
\label{p:1iibis}
There exists $\widetilde{\Omega}\in\FF$ such that 
$\PP(\widetilde{\Omega})=1$ and, for every 
$\omega\in\widetilde{\Omega}$ and every $\mathsf{z}\in\mathsf{F}$,
$(\|x_n(\omega)-\mathsf{z}\|)_{n\in\NN}$ converges.
\item
\label{p:1iii}
Suppose that $\WC(x_n)_{n\in\NN}\subset\mathsf{F}\;\:\as$ Then 
$(x_n)_{n\in\NN}$ converges weakly $\as$ to an $\mathsf{F}$-valued 
random variable.
\item
\label{p:1iv}
Suppose that $\SC(x_n)_{n\in\NN}\cap\mathsf{F}\neq\emp\;\:\as$
Then $(x_n)_{n\in\NN}$ converges strongly $\as$ to an 
$\mathsf{F}$-valued random variable.
\item
\label{p:1v}
Suppose that $\SC(x_n)_{n\in\NN}\neq\emp\;\:\as$ and that 
$\WC(x_n)_{n\in\NN}\subset\mathsf{F}\;\:\as$ Then 
$(x_n)_{n\in\NN}$ converges strongly $\as$ to an $\mathsf{F}$-valued 
random variable.
\end{enumerate}
\end{proposition}
\begin{proof}
\ref{p:1i}:
Fix $\mathsf{z}\in\mathsf{F}$. It follows from \eqref{e:sqf1} and 
Lemma~\ref{l:1971} that
$\sum_{n\in\NN}\vartheta_n(\mathsf{z})<\pinf\:\as$

\ref{p:1ii}:
Let $\mathsf{z}\in\mathsf{F}$ and
set $(\forall n\in\NN)$ $\xi_n=\|x_n-\mathsf{z}\|$.
We derive from \eqref{e:sqf1} and Lemma~\ref{l:1971} that 
$(\phi(\xi_n))_{n\in\NN}$ converges $\as$, say 
$\phi(\xi_n)\to\alpha\:\as$, where $\alpha$ is a $\RP$-valued
random variable. In turn, since 
$\lim_{t\to\pinf}\phi(t)=\pinf$, $(\xi_n)_{n\in\NN}$ is bounded 
$\as$ and so is $(x_n)_{n\in\NN}$. For subsequent use, let us also
note that 
\begin{equation}
\label{egb18T4-25a}
(\|x_n-\mathsf{z}\|)_{n\in\NN}\;
\text{converges to a $\RP$-valued random variable $\as$}
\end{equation}
Indeed, take $\omega\in\Omega$ such that 
$(\xi_n(\omega))_{n\in\NN}$ is bounded. Suppose that there
exist $\tau(\omega)\in\RP$, $\zeta(\omega)\in\RP$, and 
subsequences $(\xi_{k_n}(\omega))_{n\in\NN}$ and
$(\xi_{l_n}(\omega))_{n\in\NN}$ such that 
$\xi_{k_n}(\omega)\to\tau(\omega)$ and
$\xi_{l_n}(\omega)\to\zeta(\omega)>\tau(\omega)$, and let 
$\delta(\omega)\in
\left]0,(\zeta(\omega)-\tau(\omega))/2\right[$. Then, for $n$
sufficiently large, 
$\xi_{k_n}(\omega)\leq\tau(\omega)+\delta(\omega)<
\zeta(\omega)-\delta(\omega)\leq\xi_{l_n}(\omega)$  and, 
since $\phi$ is strictly increasing, 
$\phi(\xi_{k_n}(\omega))\leq\phi(\tau(\omega)+\delta(\omega))<
\phi(\zeta(\omega)-\delta(\omega))\leq\phi(\xi_{l_n}(\omega))$.
Taking the limit as $n\to\pinf$ yields 
$\alpha(\omega)\leq\phi(\tau(\omega)+\delta(\omega))<
\phi(\zeta(\omega)-\delta(\omega))\leq\alpha(\omega)$, 
which is impossible. It follows that $\tau(\omega)=\zeta(\omega)$
and, in turn, that $\xi_n(\omega)\to\tau(\omega)$.
Thus, $\xi_n\to\tau\;\as$

\ref{p:1iibis}: 
Since $\HH$ is separable, there exists a countable set $\mathsf{Z}$
such that $\overline{\mathsf{Z}}=\mathsf{F}$. According to 
\eqref{egb18T4-25a}, for every $\mathsf{z}\in\mathsf{F}$,
there exists a set $\Omega_{\mathsf{z}}\in\FF$ such that 
$\PP(\Omega_{\mathsf{z}})=1$ and, for every
$\omega\in\Omega_{\mathsf{z}}$, the sequence
$(\|x_{n}(\omega)-\mathsf{z}\|)_{n\in\NN}$ converges.
Now set $\widetilde{\Omega}=\bigcap_{\mathsf{z}\in\mathsf{Z}}
\Omega_{\mathsf{z}}$ and let $\complement\widetilde{\Omega}$ be
its complement. Then, since $\mathsf{Z}$ is countable,
$\PP(\widetilde{\Omega})=1-\PP(\complement\widetilde{\Omega})=
1-\PP(\bigcup_{\mathsf{z}\in\mathsf{Z}}\complement
\Omega_{\mathsf{z}})\geq 1-\sum_{\mathsf{z}\in\mathsf{Z}}
\PP(\complement\Omega_{\mathsf{z}})=1$, hence 
$\PP(\widetilde{\Omega})=1$. We now fix $\mathsf{z}\in\mathsf{F}$. 
Since $\overline{\mathsf{Z}}=\mathsf{F}$, there exists a sequence 
$(\mathsf{z}_k)_{k\in\NN}$ in $\mathsf{Z}$ such that 
$\mathsf{z}_{k}\to\mathsf{z}$. As just seen, \eqref{egb18T4-25a} 
yields 
\begin{equation}
\label{egb18T4-16u}
(\forall k\in\NN)(\exi\tau_k\colon\Omega\to\RP)
(\forall\omega\in\Omega_{\mathsf{z}_{k}})\quad
\|x_{n}(\omega)-\mathsf{z}_{k}\|\to\tau_k(\omega).
\end{equation}
Now let $\omega\in\widetilde{\Omega}$. We have
\begin{equation}
\label{egb18T4-18a}
(\forall k\in\NN)(\forall n\in\NN)\quad
-\|\mathsf{z}_k-\mathsf{z}\|\leq\|x_n(\omega)-\mathsf{z}\|-
\|x_n(\omega)-\mathsf{z}_k\|\leq\|\mathsf{z}_k-\mathsf{z}\|.
\end{equation}
Therefore
\begin{align}
\label{egb18T4-18b}
(\forall k\in\NN)\quad
-\|\mathsf{z}_k-\mathsf{z}\|
&\leq\varliminf_{n\to\pinf}\|x_n(\omega)-\mathsf{z}\|-
\lim_{n\to\pinf}\|x_n(\omega)-\mathsf{z}_k\|\nonumber\\
&=\varliminf_{n\to\pinf}\|x_n(\omega)-\mathsf{z}\|-
\tau_k(\omega)\nonumber\\
&\leq\varlimsup_{n\to\pinf}\|x_n(\omega)-\mathsf{z}\|-
\tau_k(\omega)\nonumber\\
&=\varlimsup_{n\to\pinf}\|x_n(\omega)-\mathsf{z}\|-
\lim_{n\to\pinf}\|x_n(\omega)-\mathsf{z}_k\|\nonumber\\
&\leq\|\mathsf{z}_k-\mathsf{z}\|.
\end{align}
Hence, taking the limit as $k\to\pinf$ in \eqref{egb18T4-18b},
we obtain that $(\|x_n(\omega)-\mathsf{z}\|)_{n\in\NN}$ converges;
more precisely, 
$\lim_{n\to\pinf}\|x_n(\omega)-\mathsf{z}\|=\lim_{k\to\pinf}
\tau_k(\omega)$. 

\ref{p:1iii}: 
By assumption, there exists $\widehat{\Omega}\in\FF$ such that
$\PP(\widehat{\Omega})=1$ and $(\forall\omega\in\widehat{\Omega})$ 
$\WC(x_n(\omega))_{n\in\NN}\subset\mathsf{F}$.
Now define $\widetilde{\Omega}$ as in the proof of \ref{p:1iibis},
let $\omega\in\widehat{\Omega}\cap\widetilde{\Omega}$, and let 
$x(\omega)$ and $y(\omega)$ be two points in 
$\WC(x_n(\omega))_{n\in\NN}$, say 
$x_{k_n}(\omega)\weakly x(\omega)$ and 
$x_{l_n}(\omega)\weakly y(\omega)$. Then \ref{p:1iibis} implies 
that $(\|x_n(\omega)-x(\omega)\|)_{n\in\NN}$ and 
$(\|x_n(\omega)-y(\omega)\|)_{n\in\NN}$ converge. In turn, since
\begin{equation}
(\forall n\in\NN)\quad\scal{x_n(\omega)}{x(\omega)-y(\omega)} 
=\frac12\big(\|x_n(\omega)-y(\omega)\|^2-\|x_n(\omega)-x(\omega)\|^2
+\|x(\omega)\|^2-\|y(\omega)\|^2\big),
\end{equation}
the sequence $(\scal{x_n(\omega)}{x(\omega)-y(\omega)})_{n\in\NN}$ 
converges, say 
\begin{equation}
\label{e:LNnlio07b}
\scal{x_n(\omega)}{x(\omega)-y(\omega)}\to\varrho(\omega).
\end{equation}
However, since $x_{k_n}(\omega)\weakly x(\omega)$, we have 
$\scal{x(\omega)}{x(\omega)-y(\omega)}=\varrho(\omega)$.
Likewise, passing to the limit along the subsequence 
$(x_{l_n}(\omega))_{n\in\NN}$ in \eqref{e:LNnlio07b} yields
$\scal{y(\omega)}{x(\omega)-y(\omega)}=\varrho(\omega)$. Thus,
\begin{equation}
\label{e:LNnlio08b}
0=\scal{x(\omega)}{x(\omega)-y(\omega)}-\scal{y(\omega)}
{x(\omega)-y(\omega)}=\|x(\omega)-y(\omega)\|^2.
\end{equation}
This shows that $x(\omega)=y(\omega)$. Since
$\omega\in\widetilde{\Omega}$, $(x_n(\omega))_{n\in\NN}$ is 
bounded and we invoke \cite[Lemma~2.38]{Livre1} to conclude that 
$x_n(\omega)\weakly x(\omega)\in\mathsf{F}$. Altogether, 
since $\PP(\widehat{\Omega}\cap\widetilde{\Omega})=1$,
$x_n\weakly x\;\as$ and the measurability of $x$
follows from \cite[Corollary~1.13]{Pett38}.

\ref{p:1iv}: 
Let $x\in\SC(x_n)_{n\in\NN}\cap\mathsf{F}\;\:\as$
Then there exists $\widehat{\Omega}\in\FF$ such that 
$\PP(\widehat{\Omega})=1$ and
$(\forall\omega\in\widehat{\Omega})$
$\varliminf\|x_n(\omega)-x(\omega)\|=0$.
Now let $\widetilde{\Omega}$ be as in \ref{p:1iibis} and let
$\omega\in\widetilde{\Omega}\cap\widehat{\Omega}$. Then
$\PP(\widetilde{\Omega}\cap\widehat{\Omega})=1$,
$x(\omega)\in\mathsf{F}$, and \ref{p:1iibis} implies that 
$(\|x_n(\omega)-x(\omega)\|)_{n\in\NN}$ converges. Thus,
$\lim\|x_n(\omega)-x(\omega)\|=0$. We conclude that 
$x_n\to x\;\;\as$

\ref{p:1v}$\Rightarrow$\ref{p:1iv}: 
Since $\emp\neq\SC(x_n)_{n\in\NN}\subset
\WC(x_n)_{n\in\NN}\subset\mathsf{F}\;\:\as$, we have
$\SC(x_n)_{n\in\NN}\cap\mathsf{F}\neq\emp\;\:\as$ 
\end{proof}

\begin{remark}
\label{rgb18T4-17}
Suppose that $\phi\colon t\mapsto t^2$ in \eqref{e:sqf1}. Then 
special cases of Proposition~\ref{p:1} are stated in several places 
in the literature. Thus, stochastic quasi-Fej\'er sequences were 
first discussed in \cite{Ermo69} in the case when $\HH$ is a 
Euclidean space and for every $n\in\NN$, $\vartheta_n=\chi_n=0$ 
and $\eta_n$ is deterministic. A Hilbert space version of the 
results of \cite{Ermo69} appears in \cite{Bart07} without proof. 
Finally, the case when all the processes are deterministic in 
\eqref{e:sqf1} is discussed in \cite{Else01}.
\end{remark}

The analysis of our main algorithms will rely on the following key
illustration of Proposition~\ref{p:1}. This result involves a
general stochastic iterative process and it should also be of 
interest in the analysis of the asymptotic behavior of a broad 
class of stochastic algorithms, beyond those discussed in the 
present paper.

\begin{theorem} 
\label{tHj7+1-12}
Let $\mathsf{F}$ be a nonempty closed subset of $\HH$, let 
$(\lambda_n)_{n\in \NN}$ be a sequence in $\rzeroun$, and let
$(t_{n})_{n\in\NN}$, $(x_{n})_{n\in\NN}$, and $(e_{n})_{n\in\NN}$
be sequences of $\HH$-valued random variables.
Suppose that the following hold:
\begin{enumerate}
\item
\label{tHj7+1-12ii}
$(\forall n\in\NN)$ $x_{n+1}=x_n+\lambda_n(t_n+e_n-x_n).$
\item
\label{tHj7+1-12i}
$\sum_{n\in\NN}\lambda_n\sqrt{\EC{\|e_n\|^2}{\XX_n}}<\pinf\;\;\as$
\item
\label{tHj7+1-12iii}
For every $\mathsf{z}\in\mathsf{F}$, there exist 
$(\theta_n(\mathsf{z}))_{n\in\NN}\in\ell_+({\mathscr{X}})$, 
$(\mu_n(\mathsf{z}))_{n\in\NN}\in\ell_+^\infty({\mathscr{X}})$, and 
$(\nu_n(\mathsf{z}))_{n\in\NN}\in\ell_+^\infty({\mathscr{X}})$ such 
that 
$(\lambda_n\mu_n(\mathsf{z}))_{n\in\NN}\in\ell_+^1({\mathscr{X}})$,
$(\lambda_n\nu_n(\mathsf{z}))_{n\in\NN}\in\ell_+^{1/2}
({\mathscr{X}})$, and the following is satisfied $\as$:
\begin{equation}
(\forall n\in\NN)\quad
\EC{\|t_n-\mathsf{z}\|^2}{\XX_n}+\theta_n(\mathsf{z})\leq
(1+\mu_n(\mathsf{z}))\|x_n-\mathsf{z}\|^2+\nu_n(\mathsf{z}).
\end{equation}
\end{enumerate}
Then 
\begin{equation}
\label{egb18T4-19a}
(\forall\mathsf{z}\in\mathsf{F})\quad
\bigg[\:\sum_{n\in\NN}
\lambda_n\theta_n(\mathsf{z})<\pinf\;\as\:\bigg]
\end{equation}
and
\begin{equation}
\label{egb18T4-19b}
\sum_{n\in\NN}\lambda_n(1-\lambda_n)
\EC{\|t_n-x_n\|^2}{\XX_n}<\pinf\;\as
\end{equation}
Furthermore, suppose that:
\begin{enumerate}
\setcounter{enumi}{3}
\item
\label{tHj7+1-12iv}
$\WC(x_n)_{n\in\NN}\subset\mathsf{F}\;\:\as$
\end{enumerate}
Then $(x_n)_{n\in\NN}$ converges weakly $\as$ to an 
$\mathsf{F}$-valued random variable $x$. If, in addition, 
\begin{enumerate}
\setcounter{enumi}{4}
\item
\label{tHj7+1-12v}
$\SC(x_n)_{n\in\NN}\neq\emp\;\:\as$,
\end{enumerate}
then $(x_n)_{n\in\NN}$ converges strongly $\as$ to $x$.
\end{theorem}
\begin{proof}
Let $\mathsf{z}\in\mathsf{F}$ and set 
\begin{equation}
\label{e:aprescuite}
(\forall n\in\NN)\qquad  
\varepsilon_n=\lambda_n\sqrt{\EC{\|e_n\|^2}{\XX_n}}.
\end{equation}
It follows from Jensen's inequality and \ref{tHj7+1-12iii} that
\begin{align}
\label{e:ntnzb0}
(\forall n\in\NN)\quad
\EC{\|t_n-\mathsf{z}\|}{\XX_n} 
&\leq\sqrt{\EC{\|t_n-\mathsf{z}\|^2}{\XX_n}}\nonumber\\
&\leq\sqrt{(1+\mu_n(\mathsf{z}))\|x_n-\mathsf{z}\|^2
+\nu_n(\mathsf{z})}\nonumber\\
&\leq\sqrt{1+\mu_n(\mathsf{z})}\|x_n-\mathsf{z}\|
+\sqrt{\nu_n(\mathsf{z})}\nonumber\\
&\leq(1+\mu_n(\mathsf{z})/2)\|x_n-\mathsf{z}\|
+\sqrt{\nu_n(\mathsf{z})}.
\end{align}
On the other hand, \ref{tHj7+1-12ii} and the triangle inequality 
yield
\begin{equation}
(\forall n\in\NN)\quad
\|x_{n+1}-\mathsf{z}\|\le(1-\lambda_n) \|x_n-\mathsf{z}\|+\lambda_n
\|t_n-\mathsf{z}\|+\lambda_n\|e_n\|.
\end{equation}
Consequently, 
\begin{align}
\label{e:Rob10}
(\forall n\in\NN)\quad
\EC{\|x_{n+1}-\mathsf{z}\|}{\XX_n}
&\leq\;(1-\lambda_n) \|x_n-\mathsf{z}\|+\lambda_n
\EC{\|t_n-\mathsf{z}\|}{\XX_n}+
\lambda_n\EC{\|e_n\|}{\XX_n}\nonumber\\
&\leq\Big(1+\frac{\lambda_n\mu_n(\mathsf{z})}{2}\Big)
\|x_n-\mathsf{z}\|+\lambda_n\sqrt{\nu_n(\mathsf{z})}
+\lambda_n\sqrt{\EC{\|e_n\|^2}{\XX_n}}
\nonumber\\
&=\Big(1+\frac{\lambda_n\mu_n(\mathsf{z})}{2}\Big)\|x_n-\mathsf{z}\|+
\sqrt{\lambda_n\nu_n(\mathsf{z})}+\varepsilon_n.
\end{align}
Upon applying Proposition~\ref{p:1}\ref{p:1ii} with 
$\phi\colon t\mapsto t$, we deduce from \eqref{e:Rob10} that
$(x_n)_{n\in\NN}$ is almost surely bounded and, by virtue of 
assumption~\ref{tHj7+1-12iii}, that  
$(\EC{\|t_n-\mathsf{z}\|^2}{\XX_n})_{n\in\NN}$ is likewise. Thus,
there exist $\rho_1(\mathsf{z})$ and $\rho_2(\mathsf{z})$ in $\RPP$
such that, almost surely,
\begin{equation}
\label{eHj7+1-16d}
(\forall{n\in\NN})\quad
\|x_n-\mathsf{z}\|\leq\rho_1(\mathsf{z})
\quad\text{and}\quad
\sqrt{\EC{\|t_n-\mathsf{z}\|^2}{\XX_n}}\leq\rho_2(\mathsf{z}).
\end{equation}
Now set 
\begin{equation}
\label{eHj7+1-15a}
(\forall n\in\NN)\quad 
\begin{cases}
\chi_n(\mathsf{z})=\lambda_n\mu_n(\mathsf{z})\\
\xi_n(\mathsf{z})=2\lambda_n(1-\lambda_n)\|x_n-\mathsf{z}\|
\,\|e_n\|+2\lambda_n^2\|t_n-\mathsf{z}\|\,\|e_n\|
+\lambda_n^2\|e_n\|^2\\
\vartheta_n(\mathsf{z})=\lambda_n\theta_n(\mathsf{z})+
\lambda_n(1-\lambda_n)\EC{\|t_n-x_n\|^2}{\XX_n}\\
\eta_n(\mathsf{z})=\EC{\xi_n(\mathsf{z})}{\XX_n}+
\lambda_n\nu_n(\mathsf{z}).
\end{cases}
\end{equation}
On the one hand, it follows from \eqref{eHj7+1-15a}, the
Cauchy-Schwarz inequality, and \eqref{e:aprescuite} that 
\begin{align}
\label{eHj7+1-15b}
(\forall n\in\NN)\quad\EC{\xi_n(\mathsf{z})}{\XX_n}
&=2\lambda_n(1-\lambda_n)\|x_n-\mathsf{z}\|
\,\EC{\|e_n\|}{\XX_n}\nonumber\\
&\quad\;+2\lambda_n^2\EC{\|t_n-\mathsf{z}\|\,\|e_n\|}{\XX_n}
+\lambda_n^2\EC{\|e_n\|^2}{\XX_n}\nonumber\\
&\leq 2\lambda_n\|x_n-\mathsf{z}\|
\,\sqrt{\EC{\|e_n\|^2}{\XX_n}}\nonumber\\
&\quad\;+2\lambda_n\sqrt{\EC{\|t_n-\mathsf{z}\|^2}{\XX_n}}
\sqrt{\EC{\|e_n\|^2}{\XX_n}}
+\lambda_n^2\EC{\|e_n\|^2}{\XX_n}\nonumber\\
&\leq 2(\rho_1(\mathsf{z})+\rho_2(\mathsf{z}))
\varepsilon_n+\varepsilon_n^2.
\end{align}
In turn, we deduce from \eqref{e:aprescuite}, \eqref{eHj7+1-15a},
\ref{tHj7+1-12i}, and \ref{tHj7+1-12iii} that
\begin{equation}
\label{eHj7+1-16t}
\big(\eta_n(\mathsf{z})\big)_{n\in\NN}\in\ell_+^1(\mathscr{X})
\quad\text{and}\quad
\big(\chi_n(\mathsf{z})\big)_{n\in\NN}\in\ell_+^1(\mathscr{X}).
\end{equation}
On the other hand, we derive from \ref{tHj7+1-12ii},
\cite[Corollary~2.14]{Livre1}, and \eqref{eHj7+1-15a} that
\begin{align}
\label{e:dusoir3007O}
(\forall n\in\NN)\quad\|x_{n+1}-\mathsf{z}\|^2
&=\|(1-\lambda_n) (x_n-\mathsf{z})+\lambda_n
(t_n-\mathsf{z})\|^2\nonumber\\
&\quad+2 \lambda_n\scal{(1-\lambda_n)
(x_n-\mathsf{z})+\lambda_n\big(t_n-\mathsf{z}\big)}{e_n} 
+\lambda_n^2\|e_n\|^2\nonumber\\
&\leq(1-\lambda_n)\|x_n-\mathsf{z}\|^2
+\lambda_n\|t_n-\mathsf{z}\|^2\nonumber\\
&\quad\;-\lambda_n(1-\lambda_n)
\|t_n-x_n\|^2+\xi_n(\mathsf{z}).
\end{align}
Hence, \ref{tHj7+1-12iii}, \eqref{eHj7+1-15a}, and 
\eqref{eHj7+1-15b} imply that
\begin{align}
\label{eHj7+1-16y}
&\hskip -6mm
(\forall n\in\NN)\quad
\EC{\|x_{n+1}-\mathsf{z}\|^2}{\XX_n}\nonumber\\
&\leq(1-\lambda_n)\|x_n-\mathsf{z}\|^2+
\lambda_n\EC{\|t_n-\mathsf{z}\|^2}{\XX_n}
-\lambda_n(1-\lambda_n)\EC{\|t_n-x_n\|^2}{\XX_n}
+\EC{\xi_n(\mathsf{z})}{\XX_n}\nonumber\\
&\leq(1-\lambda_n)\|x_n-\mathsf{z}\|^2+\lambda_n\big(
(1+\mu_n(\mathsf{z}))\|x_n-\mathsf{z}\|^2+\nu_n(\mathsf{z})
-\theta_n(\mathsf{z})\big)
\nonumber\\
&\quad\;-\lambda_n(1-\lambda_n)\EC{\|t_n-x_n\|^2}{\XX_n}
+\EC{\xi_n(\mathsf{z})}{\XX_n}\nonumber\\
&\leq(1+\chi_n(\mathsf{z}))\|x_n-\mathsf{z}\|^2-
\vartheta_n(\mathsf{z})+\eta_n(\mathsf{z}).
\end{align}
Thus, in view of \eqref{eHj7+1-16t}, applying 
Proposition~\ref{p:1}\ref{p:1i} with 
$\phi\colon t\mapsto t^2$ yields
$\sum_{n\in\NN}\vartheta_n(\mathsf{z})<\pinf\;\as$ and it 
follows from \eqref{eHj7+1-15a} that \eqref{egb18T4-19a}
and \eqref{egb18T4-19b} are established. Finally, the weak 
convergence assertion follows from \ref{tHj7+1-12iv} and 
Proposition~\ref{p:1}\ref{p:1iii} applied with 
$\phi\colon t\mapsto t^2$. Likewise, the strong 
convergence assertion follows from 
\ref{tHj7+1-12iv}--\ref{tHj7+1-12v} and 
Proposition~\ref{p:1}\ref{p:1v} applied with 
$\phi\colon t\mapsto t^2$.
\end{proof}

\begin{definition}
\label{dQ6t5Ds2-30}
An operator ${\mathsf T}\colon\HH\to\HH$ is 
\emph{nonexpansive} if it is $1$-Lipschitz, and \emph{demicompact}
at $\mathsf{y}\in\HH$ if for every bounded sequence
$(\mathsf{y}_n)_{n\in\NN}$ in $\HH$ such that 
$\mathsf{T}\mathsf{y}_n-\mathsf{y}_n\to\mathsf{y}$, 
we have $\SC(\mathsf{y}_n)_{n\in\NN}\neq\emp$ \cite{Petr66}.
\end{definition}

Although our primary objective is to apply 
Theorem~\ref{tHj7+1-12} to
block-coordinate methods, it also yields new results for classical
methods. As an illustration, the following application describes a
Krasnosel'ski\u\i--Mann iteration with stochastic errors. 

\begin{corollary}
\label{cgb18T2-27}
Let $(\lambda_n)_{n\in\NN}$ be a sequence in $[0,1]$ such that 
$\sum_{n\in\NN}\lambda_n(1-\lambda_n)=\pinf$ and let 
${\mathsf T}\colon\HH\to\HH$ be a nonexpansive operator such that 
$\Fix\mathsf{T}\neq\emp$. Let $x_0$ and $(e_n)_{n\in\NN}$ 
be $\HH$-valued random variables. Iterate
\begin{equation}
\label{egb18T2-09x}
\begin{array}{l}
\text{for}\;n=0,1,\ldots\\
\left\lfloor
\begin{array}{l}
x_{n+1}=x_n+\lambda_n\big(\mathsf{T}x_n+e_n-x_n\big).
\end{array} 
\right.
\end{array} 
\end{equation}
In addition, assume that $\sum_{n\in\NN}\lambda_n
\sqrt{\EC{\|e_n\|^2}{\XX_n}}<\pinf$ $\as$
Then the following hold:
\begin{enumerate}
\item
\label{cgb18T2-27i}
$(x_n)_{n\in\NN}$ converges weakly $\as$ to a 
$(\Fix{\mathsf T})$-valued random variable.
\item
\label{cgb18T2-27ii}
Suppose that $\mathsf{T}$ is demicompact at $\mathsf{0}$
(see Definition~\ref{dQ6t5Ds2-30}).
Then $(x_n)_{n\in\NN}$ converges strongly $\as$ to a 
$(\Fix{\mathsf T})$-valued random variable.
\end{enumerate}
\end{corollary}
\begin{proof}
Set $\mathsf{F}=\Fix\mathsf{T}$. Since $\mathsf{T}$ is continuous, 
$\mathsf{T}$ is measurable and $\mathsf{F}$ is closed. Now let 
$\mathsf{z}\in\mathsf{F}$ and set $(\forall n\in\NN)$ 
$t_n=\mathsf{T}x_n$. Then, using the nonexpansiveness of 
$\mathsf{T}$, we obtain
\begin{equation}
\label{egb18T2-28b}
(\forall n\in\NN)\quad 
\begin{cases}
x_{n+1}=x_n+\lambda_n(t_n+e_n-x_n)\\
\EC{\|t_n-x_n\|^2}{\XX_n}=\|\mathsf{T}x_n-x_n\|^2\\
\EC{\|t_n-\mathsf{z}\|^2}{\XX_n}=
\|\mathsf{T}x_n-\mathsf{T}\mathsf{z}\|^2
\leq\|x_n-\mathsf{z}\|^2.
\end{cases}
\end{equation}
It follows that properties 
\ref{tHj7+1-12ii}--\ref{tHj7+1-12iii} in 
Theorem~\ref{tHj7+1-12} are satisfied with 
$(\forall n\in\NN)$ $\theta_n=0$, $\mu_n=0$, and $\nu_n=0$. Hence,
\eqref{egb18T4-19b} and \eqref{egb18T2-28b} imply the existence 
of $\widetilde{\Omega}\in\FF$ such that $\PP(\widetilde{\Omega})=1$
and
\begin{equation}
\label{egb18T2-28c}
(\forall\omega\in\widetilde{\Omega})\quad
\sum_{n\in\NN}\lambda_n(1-\lambda_n)
\|\mathsf{T}x_n(\omega)-x_n(\omega)\|^2<\pinf.
\end{equation}
Moreover,
\begin{align}
\label{egb18T2-20a}
(\forall n\in\NN)\quad
\|{\mathsf T}{x}_{n+1}-{x}_{n+1}\|
&=\|{\mathsf T}{x}_{n+1}-{\mathsf T}{x}_n
+(1-\lambda_n)({\mathsf T}{x}_n-{x}_n)
-\lambda_ne_n\|\nonumber\\
&\leq\|{\mathsf T}{x}_{n+1}-{\mathsf T}{x}_n\|+
(1-\lambda_n)\|{\mathsf T}{x}_n-{x}_n\|+\lambda_n\|e_n\|
\nonumber\\
&\leq\|{x}_{n+1}-{x}_n\|+(1-\lambda_n)
\|{\mathsf T}{x}_n-{x}_n\|+\lambda_n\|e_n\|
\nonumber\\
&\leq\lambda_n\|{\mathsf T}{x}_n-{x}_n\|+(1-\lambda_n)
\|{\mathsf T}{x}_n-{x}_n\|+2\lambda_n\|e_n\|
\nonumber\\
&=\|{\mathsf T}{x}_n-{x}_n\|+2\lambda_n\|e_n\|
\end{align}
and, therefore,
\begin{align}
(\forall n\in\NN)\quad
\EC{\|{\mathsf T}{x}_{n+1}-{x}_{n+1}\|}{\XX_n}
&\leq\|{\mathsf T}{x}_n-{x}_n\|+2\lambda_n \EC{\|e_n\|}{\XX_n}
\nonumber\\
&\leq\|{\mathsf T}{x}_n-{x}_n\|+2\lambda_n
\sqrt{\EC{\|e_n\|^2}{\XX_n}}.
\end{align}
In turn, Lemma~\ref{l:1971} implies that there exists 
$\widehat{\Omega}\subset\widetilde{\Omega}$ such that
$\widehat{\Omega}\in\FF$, $\PP(\widehat{\Omega})=1$, and,
for every $\omega\in\widehat{\Omega}$,
$(\|\mathsf{T}x_n(\omega)-x_n(\omega)\|)_{n\in\NN}$ converges. 

\ref{cgb18T2-27i}:
It is enough to establish property \ref{tHj7+1-12iv} of 
Theorem~\ref{tHj7+1-12}. Let $\omega\in\widehat{\Omega}$ and 
let $\mathsf{x}\in\WC(x_n(\omega))_{n\in\NN}$, say 
$x_{k_n}(\omega)\weakly\mathsf{x}$. In view of 
\eqref{egb18T2-28c}, since 
$\sum_{n\in\NN}\lambda_n(1-\lambda_n)=\pinf$, we have
$\varliminf\|\mathsf{T}x_n(\omega)-x_n(\omega)\|=0$. Therefore,
\begin{equation}
\label{eQ6t5Ds2-28a}
\|\mathsf{T}x_n(\omega)-x_n(\omega)\|\to 0. 
\end{equation}
Altogether, $x_{k_n}(\omega)\weakly\mathsf{x}$
and $\mathsf{T}x_{k_n}(\omega)-x_{k_n}(\omega)\to 0$.
Since $\mathsf{T}$ is nonexpansive, the demiclosed 
principle \cite[Corollary~4.18]{Livre1} asserts that 
$\mathsf{x}\in\mathsf{F}$.

\ref{cgb18T2-27ii}:
It is enough to establish property \ref{tHj7+1-12v} of 
Theorem~\ref{tHj7+1-12}. Let $\omega\in\widehat{\Omega}$. 
As shown above, $(x_n(\omega))_{n\in\NN}$ converges weakly and
it is therefore bounded \cite[Lemma~2.38]{Livre1}. Hence, by
demicompactness, \eqref{eQ6t5Ds2-28a} implies that 
$\SC(x_n(\omega))_{n\in\NN}\neq\emp$.
Thus, $\SC(x_n)_{n\in\NN}\neq\emp\;\:\as$
\end{proof}

\begin{remark}
\label{rgb18T2-28}
Corollary~\ref{cgb18T2-27} extends \cite[Theorem~5.5]{Else01}, 
which is restricted to deterministic processes and therefore less 
realistic error models. As shown in \cite{Livre1,Cegi12,Else01}, 
the Krasnosel'ski\u\i--Mann iteration process is at the core of 
many algorithms in variational problems and optimization. 
Corollary~\ref{cgb18T2-27} therefore provides stochastically 
perturbed versions of these algorithms.  
\end{remark}

\section{Single-layer random block-coordinate fixed point algorithms}
\label{sec:3}

In the remainder of the paper, the following notation will be used.

\begin{notation}
\label{n:2}
$\HH_1,\ldots,\HH_m$ are separable
real Hilbert spaces and $\HHH=\HH_1\oplus\cdots\oplus\HH_m$ is 
their direct Hilbert sum. The scalar products and associated norms 
of these spaces are all denoted by 
$\scal{\cdot}{\cdot}$ and $\|\cdot\|$, respectively, and
$\boldsymbol{\mathsf x}=({\mathsf x}_1,\ldots,{\mathsf x}_m)$
denotes a generic vector in $\HHH$. Given a sequence 
$(\boldsymbol{x}_n)_{n\in\NN}=
(x_{1,n},\ldots,x_{m,n})_{n\in\NN}$ of $\HHH$-valued random 
variables, we set $(\forall n\in\NN)$
$\XXX_n=\sigma(\boldsymbol{x}_0,\ldots,\boldsymbol{x}_n)$. 
\end{notation}

We recall that an operator $\boldsymbol{\mathsf{T}}\colon\HHH\to\HHH$
with fixed point set $\Fix\boldsymbol{\mathsf{T}}$ is 
quasinonexpansive if \cite{Livre1}
\begin{equation}
\label{e:quasi}
(\forall\boldsymbol{\mathsf{z}}\in\Fix
\boldsymbol{\mathsf{T}})(\forall \boldsymbol{\mathsf{x}}\in\HHH)
\quad\|\boldsymbol{\mathsf{T}}\boldsymbol{\mathsf{x}}-
\boldsymbol{\mathsf{z}}\|\leq\|\boldsymbol{\mathsf{x}}-
\boldsymbol{\mathsf{z}}\|.
\end{equation}

\begin{theorem}
\label{tgb18T2-12}
Let $(\lambda_n)_{n\in\NN}$ be a sequence in $\zeroun$ such that
$\inf_{n\in\NN}\lambda_n>0$ and $\sup_{n\in\NN}\lambda_n<1$
and set $\mathsf{D}=\{0,1\}^m\smallsetminus
\{\boldsymbol{\mathsf{0}}\}$. For every $n\in\NN$, let 
$\boldsymbol{\mathsf T}_{\!n}\colon\HHH\to\HHH\colon
\boldsymbol{\mathsf x}\mapsto({\mathsf T}_{\!i,n}\,
\boldsymbol{\mathsf x})_{1\leq i\leq m}$ be a quasinonexpansive 
operator where, for every $i\in\{1,\ldots,m\}$, 
${\mathsf T}_{\!i,n}\colon\HHH\to\HH_i$ is measurable.
Let $\boldsymbol{x}_0$ and $(\boldsymbol{a}_n)_{n\in\NN}$ 
be $\HHH$-valued random variables, and let 
$(\boldsymbol{\varepsilon}_n)_{n\in\NN}$ be identically distributed 
$\mathsf{D}$-valued random variables. Iterate
\begin{equation}
\label{egb18T2-09}
\begin{array}{l}
\text{for}\;n=0,1,\ldots\\
\left\lfloor
\begin{array}{l}
\text{for}\;i=1,\ldots,m\\
\left\lfloor
\begin{array}{l}
x_{i,n+1}=x_{i,n}+\varepsilon_{i,n}\lambda_n\big(
{\mathsf T}_{\!i,n}\,(x_{1,n},\ldots,x_{m,n})+a_{i,n}-x_{i,n}\big),
\end{array} 
\right.
\end{array} 
\right.
\end{array} 
\end{equation}
and set $(\forall n\in\NN)$ 
$\EEE_n=\sigma(\boldsymbol{\varepsilon}_n)$.
In addition, assume that the following hold:
\begin{enumerate}
\item
\label{tgb18T2-12i}
$\boldsymbol{\mathsf F}=\bigcap_{n\in\NN}
\Fix\boldsymbol{\mathsf T}_{\!n}\neq\emp$.
\item
\label{tgb18T2-12ii}
$\sum_{n\in\NN}\sqrt{\EC{\|\boldsymbol{a}_n\|^2}{\XXX_n}}<\pinf$.
\item
\label{tgb18T2-12iv}
For every $n\in\NN$, $\EEE_n$ and $\XXX_n$ are independent.
\item
\label{tgb18T2-12v}
$(\forall i\in\{1,\ldots,m\})$ 
$\mathsf{p}_i=\PP[\varepsilon_{i,0}=1]>0$.
\end{enumerate}
Then
\begin{equation}
\label{eQ6t5Ds2-30c}
\boldsymbol{\mathsf T}_{\!n}\boldsymbol{x}_n-
\boldsymbol{x}_n\to\boldsymbol{0}\;\:\as
\end{equation}
Furthermore, suppose that:
\begin{enumerate}
\setcounter{enumi}{4}
\item
\label{tgb18T2-12iii}
$\WC(\boldsymbol{x}_n)_{n\in\NN}\subset\boldsymbol{\mathsf{F}}
\;\:\as$
\end{enumerate}
Then $(\boldsymbol{x}_n)_{n\in\NN}$ converges weakly $\as$ to an 
$\boldsymbol{\mathsf{F}}$-valued random variable $\boldsymbol{x}$. 
If, in addition, 
\begin{enumerate}
\setcounter{enumi}{5}
\item
\label{tgb18T2-12vi}
$\SC(\boldsymbol{x}_n)_{n\in\NN}\neq\emp\;\:\as$,
\end{enumerate}
then $(\boldsymbol{x}_n)_{n\in\NN}$ converges strongly 
$\as$ to $\boldsymbol{x}$.
\end{theorem}
\begin{proof}
We define a norm $|||\cdot|||$ on $\HHH$ by
\begin{equation}
\label{egb18T2-14}
(\forall\boldsymbol{\mathsf{x}}\in\HHH)\quad
|||\boldsymbol{\mathsf{x}}|||^2=\sum_{i=1}^m\frac{1}{\mathsf{p}_i}
\|\mathsf{x}_i\|^2.
\end{equation}
We are going to apply Theorem~\ref{tHj7+1-12} in 
$(\HHH,|||\cdot|||)$. Let us set
\begin{equation}
\label{egb18T2-16}
(\forall n\in\NN)\quad
\begin{cases}
\boldsymbol{t}_n=(t_{i,n})_{1\leq i\leq m}\\
\boldsymbol{e}_n=(\varepsilon_{i,n}a_{i,n})_{1\leq i\leq m},
\end{cases}
\,\text{where}\;
(\forall i\in\{1,\ldots,m\})\;
t_{i,n}=x_{i,n}+\varepsilon_{i,n}({\mathsf T}_{\!i,n}\,
\boldsymbol{x}_n-x_{i,n}).
\end{equation}
Then it follows from \eqref{egb18T2-09} that
\begin{equation}
\label{egb18T2-17a}
(\forall n\in\NN)\quad
\boldsymbol{x}_{n+1}=\boldsymbol{x}_{n}+\lambda_n\big(
\boldsymbol{t}_n+\boldsymbol{e}_n-\boldsymbol{x}_{n}\big),
\end{equation}
while \ref{tgb18T2-12ii} implies that
\begin{equation}
\label{egb18T2-17b}
\sum_{n\in\NN}\lambda_n\EC{|||\boldsymbol{e}_n|||^2}{\XXX_n}\leq
\sum_{n\in\NN}\EC{|||\boldsymbol{a}_n|||^2}{\XXX_n}<\pinf.
\end{equation}
Since the operators $(\boldsymbol{\mathsf{T}}_{\!n})_{n\in \NN}$ are 
quasinonexpansive, $\boldsymbol{\mathsf{F}}$ is closed 
\cite[Section~2]{Moor01}. Now let 
$\boldsymbol{\mathsf{z}}\in\boldsymbol{\mathsf{F}}$ and set
\begin{equation}
\label{egb18T4-13a}
(\forall n\in\NN)(\forall i\in\{1,\ldots,m\})\quad
\mathsf{q}_{i,n}\colon\HHH\times\mathsf{D}\to\RR\colon
(\boldsymbol{\mathsf{x}},\boldsymbol{\epsilon})\mapsto
\|\mathsf{x}_i-\mathsf{z}_i+\epsilon_i({\mathsf T}_{\!i,n}\,
\boldsymbol{\mathsf{x}}-\mathsf{x}_i)\|^2.
\end{equation}
Note that, for every $n\in\NN$ and every $i\in\{1,\ldots,m\}$, 
since ${\mathsf T}_{\!i,n}$ is measurable, so are the functions 
$(\mathsf{q}_{i,n}(\cdot,\boldsymbol{\epsilon}))%
_{\boldsymbol{\epsilon}\in\mathsf{D}}$. Consequently, since, 
for every $n\in\NN$, \ref{tgb18T2-12iv} asserts that the events 
$([\boldsymbol{\varepsilon}_n=
\boldsymbol{\epsilon}])_{\boldsymbol{\epsilon}\in\mathsf{D}}$ form 
an almost sure partition of $\Omega$ and are independent 
from $\XXX_n$, and since the random variables 
$(\mathsf{q}_{i,n}(\boldsymbol{x}_n,\boldsymbol{\epsilon}))%
_{\substack{1\leq i\leq m\\ \boldsymbol{\epsilon}\in\mathsf{D}}}$ 
are $\XXX_n$-measurable, we obtain \cite[Section~28.2]{LoevII} 
\begin{align}
\label{egb18T4-15c}
(\forall n\in\NN)(\forall i\in\{1,\ldots,m\})\quad
\EC{\|t_{i,n}-\mathsf{z}_{i}\|^2}{\XXX_n}
&=\EEC{\mathsf{q}_{i,n}(\boldsymbol{x}_n,\boldsymbol{\varepsilon}_n)
\sum_{\boldsymbol{\epsilon}\in\mathsf{D}}
1_{[\boldsymbol{\varepsilon}_n=\boldsymbol{\epsilon}]}}{\XXX_n}
\nonumber\\
&=\sum_{\boldsymbol{\epsilon}\in\mathsf{D}}
\EC{\mathsf{q}_{i,n}(\boldsymbol{x}_n,\boldsymbol{\epsilon})
1_{[\boldsymbol{\varepsilon}_n=\boldsymbol{\epsilon}]}}{\XXX_n}
\nonumber\\
&=\sum_{\boldsymbol{\epsilon}\in\mathsf{D}}
\EC{1_{[\boldsymbol{\varepsilon}_n=\boldsymbol{\epsilon}]}}{\XXX_n}
\mathsf{q}_{i,n}(\boldsymbol{x}_n,\boldsymbol{\epsilon})\nonumber\\
&=\sum_{\boldsymbol{\epsilon}\in\mathsf{D}}
\PP[\boldsymbol{\varepsilon}_n=\boldsymbol{\epsilon}]
\mathsf{q}_{i,n}(\boldsymbol{x}_n,\boldsymbol{\epsilon}).
\end{align}
Thus, \eqref{egb18T2-14}, \eqref{egb18T2-16}, 
\eqref{egb18T4-15c}, \eqref{egb18T4-13a}, \ref{tgb18T2-12v}, 
and \eqref{e:quasi} yield
\begin{align}
\label{egb18T2-18}
&\hskip -6mm
(\forall n\in\NN)\quad
\EC{|||\boldsymbol{t}_n-\boldsymbol{\mathsf{z}}|||^2}
{\XXX_n}\nonumber\\
&=\sum_{i=1}^m\frac{1}{\mathsf{p}_i}
\EC{\|t_{i,n}-\mathsf{z}_{i}\|^2}{\XXX_n}\nonumber\\
&=\sum_{i=1}^m\frac{1}{\mathsf{p}_i}
\sum_{\boldsymbol{\epsilon}\in\mathsf{D}}
\PP[\boldsymbol{\varepsilon}_n=\boldsymbol{\epsilon}]
\|x_{i,n}-\mathsf{z}_i+\epsilon_{i}({\mathsf T}_{\!i,n}\,
\boldsymbol{x}_n-x_{i,n})\|^2
\nonumber\\
&=\sum_{i=1}^m\frac{1}{\mathsf{p}_i}
\left(\sum_{\boldsymbol{\epsilon}\in\mathsf{D},\epsilon_i=1}
\PP[\boldsymbol{\varepsilon}_n=\boldsymbol{\epsilon}]\,
\|\mathsf{T}_{\!i,n}\,\boldsymbol{x}_n-\mathsf{z}_{i}\|^2+
\sum_{\substack{\boldsymbol{\epsilon}\in\mathsf{D},\,\epsilon_i=0}}
\PP[\boldsymbol{\varepsilon}_n=\boldsymbol{\epsilon}]\,
\|x_{i,n}-\mathsf{z}_{i}\|^2\right)\nonumber\\
&=\|\boldsymbol{\mathsf{T}_{\!n}}\boldsymbol{x}_n-
\boldsymbol{\mathsf{z}}\|^2+
\sum_{i=1}^m \frac{1-\mathsf{p}_i}{\mathsf{p}_i} 
\|x_{i,n}-\mathsf{z}_{i}\|^2\nonumber\\
&=|||\boldsymbol{x}_n-\boldsymbol{\mathsf z}|||^2+
\|\boldsymbol{\mathsf T}_{\!n}\boldsymbol{x}_n-
\boldsymbol{\mathsf z}\|^2
-\|\boldsymbol{x}_n-\boldsymbol{\mathsf z}\|^2\nonumber\\
&\leq|||\boldsymbol{x}_n-\boldsymbol{\mathsf z}|||^2.
\end{align}
Altogether, properties \ref{tHj7+1-12ii}--\ref{tHj7+1-12iii} 
of Theorem~\ref{tHj7+1-12} are satisfied with $(\forall n\in\NN)$
$\theta_n=\mu_n=\nu_n=0$. We therefore derive from 
\eqref{egb18T4-19b} that 
$\sum_{n\in\NN}\lambda_n(1-\lambda_n)
\EC{|||\boldsymbol{t}_n-\boldsymbol{x}_n|||^2}{\XXX_n}<\pinf\;\as$
In view of our conditions on $(\lambda_n)_{n\in\NN}$, this yields
\begin{equation}
\label{egb18T3-04b}
\EC{|||\boldsymbol{t}_n-\boldsymbol{x}_n|||^2}{\XXX_n}\to 0\;\:\as
\end{equation}
On the other hand, proceeding as in \eqref{egb18T4-15c} leads to
\begin{equation}
\label{egb18T4-13x}
(\forall n\in\NN)(\forall i\in\{1,\ldots,m\})\quad
\EC{\|t_{i,n}-x_{i,n}\|^2}{\XXX_n}=
\sum_{\boldsymbol{\epsilon}\in\mathsf{D}}
\epsilon_i\PP[\boldsymbol{\varepsilon}_n=\boldsymbol{\epsilon}]
\|{\mathsf T}_{\!i,n}\,\boldsymbol{x}_n-x_{i,n}\|^2.
\end{equation}
Hence, it follows from \eqref{egb18T2-14},
\eqref{egb18T2-16}, and \ref{tgb18T2-12v} that
\begin{align}
\label{egb18T2-27}
(\forall n\in\NN)\quad
\EC{|||\boldsymbol{t}_n-\boldsymbol{x}_n|||^2}
{\boldsymbol{\XX}_n}
&=\sum_{i=1}^m\frac{1}{\mathsf{p}_i}\EC{\|t_{i,n}-x_{i,n}\|^2}
{\boldsymbol{\XX}_n}\nonumber\\
&=\sum_{i=1}^m\frac{1}{\mathsf{p}_i}
\sum_{\boldsymbol{\epsilon}\in\mathsf{D}}\epsilon_i
\PP[\boldsymbol{\varepsilon}_n=\boldsymbol{\epsilon}]
\|{\mathsf T}_{\!i,n}\,\boldsymbol{x}_n-x_{i,n}\|^2
\nonumber\\
&=\sum_{i=1}^m\frac{1}{\mathsf{p}_i}
\sum_{\boldsymbol{\epsilon}\in\mathsf{D},\epsilon_i=1}
\PP[\boldsymbol{\varepsilon}_n=\boldsymbol{\epsilon}]
\|{\mathsf T}_{\!i,n}\,\boldsymbol{x}_n-x_{i,n}\|^2
\nonumber\\
&=\|\boldsymbol{\mathsf T}_{\!n}\boldsymbol{x}_n-
\boldsymbol{x}_n\|^2.
\end{align}
Accordingly, \eqref{egb18T3-04b} yields 
$\boldsymbol{\mathsf T}_{\!n}\boldsymbol{x}_n-
\boldsymbol{x}_n\to\boldsymbol{0}\;\:\as$
In turn, the weak and strong convergence assertions are 
consequences of Theorem~\ref{tHj7+1-12}.
\end{proof}

\begin{remark} 
\label{r:match}
Let us make a few comments about Theorem~\ref{tgb18T2-12}.
\begin{enumerate}
\item 
The binary variable $\varepsilon_{i,n}$ signals whether the 
$i$-th coordinate $\mathsf{T}_{\!i,n}$ of the operator
$\boldsymbol{\mathsf{T}}_{\!n}$ is activated or not at iteration 
$n$.
\item 
\label{r:matchi} 
Assumption~\ref{tgb18T2-12v} guarantees that each operator in 
$({\mathsf T}_{\!i,n})_{1\leq i\leq m}$ is activated with a 
nonzero probability at each iteration $n$ of 
Algorithm~\eqref{e:main1}. The simplest scenario corresponds to 
the case when the block sweeping process assigns nonzero 
probabilities to multivariate indices 
$\boldsymbol{\epsilon}\in\mathsf{D}$ 
having a single component equal to 1. Then only one of the 
operators in $({\mathsf T}_{\!i,n})_{1\leq i\leq m}$ is activated 
randomly. In general, the sweeping rule allows for an arbitrary
sampling of the indices $\{1,\ldots,m\}$.
\item 
\label{r:matchiii} 
In view of \eqref{eQ6t5Ds2-30c}, \ref{tgb18T2-12iii} is 
satisfied if there exists $\widehat{\Omega}\in\FF$ such that
$\PP(\widehat{\Omega})=1$ and
\begin{equation}
\label{eQ6t5Ds2-30a}
(\forall\omega\in\widehat{\Omega})\quad
\Big[\,\boldsymbol{\mathsf T}_{\!n}\boldsymbol{x}_n(\omega)-
\boldsymbol{x}_n(\omega)\to\boldsymbol{0}\quad\Rightarrow\quad
\WC(\boldsymbol{x}_n(\omega))_{n\in\NN}
\subset\boldsymbol{\mathsf{F}}\,\Big].
\end{equation}
In the deterministic case, this is akin to the focusing 
conditions of \cite{Baus96}; see \cite{Baus96,Moor01,Else01}
for examples of suitable sequences 
$(\boldsymbol{\mathsf{T}}_{\!n})_{n\in\NN}$.
Likewise, \ref{tgb18T2-12vi} is satisfied if
there exists $\widehat{\Omega}\in\FF$ such that
$\PP(\widehat{\Omega})=1$ and
\begin{equation}
\label{eQ6t5Ds2-30b}
(\forall\omega\in\widehat{\Omega})\quad
\bigg[\,\bigg[\,\sup_{n\in\NN}\|\boldsymbol{x}_n(\omega)\|
<\pinf\quad\text{and}
\quad\boldsymbol{\mathsf T}_{\!n}\boldsymbol{x}_n(\omega)-
\boldsymbol{x}_n(\omega)\to\boldsymbol{0}\,\bigg]
\;\;\Rightarrow\;\;
\SC(\boldsymbol{x}_n(\omega))_{n\in\NN}\neq\emp\,\bigg].
\end{equation}
In the deterministic case, this is the demicompactness 
regularity condition of \cite[Definition~6.5]{Else01}.
Examples of suitable sequences 
$(\boldsymbol{\mathsf{T}}_{\!n})_{n\in\NN}$ 
are provided in \cite{Else01}.
\end{enumerate}
\end{remark}

Our first corollary is a random block-coordinate version of the 
Krasnosel'ski\u\i--Mann iteration.

\begin{corollary}
\label{cgb18T3-04}
Let $(\lambda_n)_{n\in\NN}$ be a sequence in $\zeroun$ such that
$\inf_{n\in\NN}\lambda_n>0$ and $\sup_{n\in\NN}\lambda_n<1$, set 
$\mathsf{D}=\{0,1\}^m\smallsetminus\{\boldsymbol{\mathsf{0}}\}$, 
and let $\boldsymbol{\mathsf T}\colon\HHH\to\HHH\colon
\boldsymbol{\mathsf x}\mapsto({\mathsf T}_{\!i}\,
\boldsymbol{\mathsf x})_{1\leq i\leq m}$ be a nonexpansive operator
such that $\Fix\boldsymbol{\mathsf T}\neq\emp$ where, for every 
$i\in\{1,\ldots,m\}$, ${\mathsf T}_{\!i}\colon\HHH\to\HH_i$.
Let $\boldsymbol{x}_0$ and $(\boldsymbol{a}_n)_{n\in\NN}$ be 
$\HHH$-valued random variables, and let 
$(\boldsymbol{\varepsilon}_n)_{n\in\NN}$ be identically distributed
$\mathsf{D}$-valued random variables. Iterate
\begin{equation}
\label{egb18T3-05a}
\begin{array}{l}
\text{for}\;n=0,1,\ldots\\
\left\lfloor
\begin{array}{l}
\text{for}\;i=1,\ldots,m\\
\left\lfloor
\begin{array}{l}
x_{i,n+1}=x_{i,n}+\varepsilon_{i,n}\lambda_n\big({\mathsf T}_{\!i}\,
(x_{1,n},\ldots,x_{m,n})+a_{i,n}-x_{i,n}\big),
\end{array} 
\right.
\end{array} 
\right.
\end{array} 
\end{equation}
and set $(\forall n\in\NN)$ 
$\EEE_n=\sigma(\boldsymbol{\varepsilon}_n)$. In addition, assume 
that properties \ref{tgb18T2-12ii}--\ref{tgb18T2-12v} of 
Theorem~\ref{tgb18T2-12} hold.
Then $(\boldsymbol{x}_n)_{n\in\NN}$ converges weakly
$\as$ to a $(\Fix\boldsymbol{\mathsf T})$-valued random variable.
The convergence is strong if $\boldsymbol{\mathsf T}$ is 
demicompact at $\mathsf{0}$ (see Definition~\ref{dQ6t5Ds2-30}).
\end{corollary}
\begin{proof}
This is an application of Theorem~\ref{tgb18T2-12} with
$\boldsymbol{\mathsf{F}}=\Fix\boldsymbol{\mathsf{T}}$ and 
$(\forall n\in\NN)$ 
$\boldsymbol{\mathsf{T}}_{\!n}=\boldsymbol{\mathsf{T}}$. Indeed,
\eqref{eQ6t5Ds2-30a} follows from the demiclosed principle 
\cite[Corollary~4.18]{Livre1} and \eqref{eQ6t5Ds2-30b} follows from
the demicompactness assumption.
\end{proof}

\begin{remark}
A special case of Corollary~\ref{cgb18T3-04} appears in 
\cite{Iutz13}. It corresponds to the scenario in which $\HHH$ is 
finite-dimensional, $\boldsymbol{\mathsf T}$ is firmly nonexpansive,
and, for every $n\in\NN$, $\lambda_n=1$, 
$\boldsymbol{a}_n=\boldsymbol{0}$, and only one block is activated
as in Remark~\ref{r:match}\ref{r:matchi}. 
Let us also note that a renorming similar to that performed 
in \eqref{egb18T2-14} was employed in \cite{Nest12}.
\end{remark}

Next, we consider the construction of a fixed point of a family of 
averaged operators.

\begin{definition}
\label{d:averaged}
Let $\mathsf{T}\colon\HH\to\HH$ be nonexpansive and let 
$\alpha\in\zeroun$. Then $\mathsf{T}$ is \emph{averaged} 
with constant $\alpha$, or $\alpha$-averaged, if there exists 
a nonexpansive operator $\mathsf{R}\colon\HH\to\HH$ such that
$\mathsf{T}=(1-\alpha)\Id+\alpha\mathsf{R}$.
\end{definition}

\begin{proposition}{\rm\cite[Proposition~4.25]{Livre1}}
\label{p:av1}
Let $\mathsf{T}\colon\HH\to\HH$ be nonexpansive and let 
$\alpha\in\zeroun$. Then 
$\mathsf{T}$ is $\alpha$-averaged if and only if
$(\forall \mathsf{x}\in\HH)(\forall \mathsf{y}\in\HH)$ 
$\|\mathsf{T}\mathsf{x}-\mathsf{T}\mathsf{y}\|^2\leq
\|\mathsf{x}-\mathsf{y}\|^2-\displaystyle{\frac{1-\alpha}{\alpha}}
\|(\Id-\mathsf{T})\mathsf{x}-(\Id-\mathsf{T})\mathsf{y}\|^2$.
\end{proposition}

\begin{corollary}
\label{cgb18T3-03}
Let $\chi\in\zeroun$, let $(\alpha_n)_{n\in\NN}$ be a sequence in 
$\zeroun$, and set 
$\mathsf{D}=\{0,1\}^m\smallsetminus\{\boldsymbol{\mathsf{0}}\}$. 
For every $n\in\NN$, let $\lambda_n\in\left[\chi/\alpha_n,%
(1-\chi)/\alpha_n\right]$ and let
$\boldsymbol{\mathsf T}_{\!n}\colon\HHH\to\HHH\colon
\boldsymbol{\mathsf x}\mapsto({\mathsf T}_{\!{i,n}}\,
\boldsymbol{\mathsf x})_{1\leq i\leq m}$ be an $\alpha_n$-averaged
operator, where, for every $i\in\{1,\ldots,m\}$, 
${\mathsf T}_{\!i,n}\colon\HHH\to\HH_i$.
Furthermore, let $\boldsymbol{x}_0$ and 
$(\boldsymbol{a}_n)_{n\in\NN}$ 
be $\HHH$-valued random variables, and let 
$(\boldsymbol{\varepsilon}_n)_{n\in\NN}$ be identically 
distributed $\mathsf{D}$-valued random variables. Iterate
\begin{equation}
\label{egb18T3-02}
\begin{array}{l}
\text{for}\;n=0,1,\ldots\\
\left\lfloor
\begin{array}{l}
\text{for}\;i=1,\ldots,m\\
\left\lfloor
\begin{array}{l}
x_{i,n+1}= 
x_{i,n}+\varepsilon_{i,n}\lambda_n\big({\mathsf T}_{\!{i,n}}\,
(x_{1,n},\ldots,x_{m,n})+a_{i,n}-x_{i,n}\big),
\end{array} 
\right.
\end{array} 
\right.
\end{array} 
\end{equation}
and set $(\forall n\in\NN)$ 
$\EEE_n=\sigma(\boldsymbol{\varepsilon}_n)$.
Furthermore, assume that there exists $\widehat{\Omega}\in\FF$ 
such that $\PP(\widehat{\Omega})=1$ and the following hold:
\begin{enumerate}
\item
\label{cgb18T3-03i}
$\boldsymbol{\mathsf F}=\bigcap_{n\in\NN}\Fix
\boldsymbol{\mathsf T}_{\!n}\neq\emp$.
\item
\label{cgb18T3-03ii}
$\sum_{n\in\NN}\alpha_n^{-1}\sqrt{\EC{\|\boldsymbol{a}_n\|^2}
{\XXX_n}}<\pinf$.
\item
\label{cgb18T3-03iv-}
For every $n\in\NN$, $\EEE_n$ and $\XXX_n$ are independent.
\item
\label{cgb18T3-03iv}
$(\forall i\in\{1,\ldots,m\})$ $\PP[\varepsilon_{i,0}=1]>0$.
\item
\label{cgb18T3-03iii}
$(\forall\omega\in\widehat{\Omega})$
$\Big[\alpha_n^{-1}\big(\boldsymbol{\mathsf T}_{\!n}
\boldsymbol{x}_{n}
(\omega)-\boldsymbol{x}_{n}(\omega)\big)\to\boldsymbol{0}$
$\Rightarrow$ $\WC(\boldsymbol{x}_n(\omega))_{n\in\NN}
\subset\boldsymbol{\mathsf{F}}\Big]$.
\end{enumerate}
Then $(\boldsymbol{x}_n)_{n\in\NN}$ converges weakly
$\as$ to an $\boldsymbol{\mathsf{F}}$-valued random variable
$\boldsymbol{x}$. If, in addition, 
\begin{enumerate}
\setcounter{enumi}{5}
\item
$(\forall\omega\in\widehat{\Omega})$
$\Big[\,\big[\,\sup_{n\in\NN}\|\boldsymbol{x}_n(\omega)\|<\pinf$ 
and 
$\alpha_n^{-1}\big(\boldsymbol{\mathsf T}_{\!n}
\boldsymbol{x}_n(\omega)-\boldsymbol{x}_n(\omega)\big)\to
\boldsymbol{0}\,\big]$ $\Rightarrow$
$\SC(\boldsymbol{x}_n(\omega))_{n\in\NN}\neq\emp\,\Big]$,
\end{enumerate}
then $(\boldsymbol{x}_n)_{n\in\NN}$ converges strongly
$\as$ to $\boldsymbol{x}$.
\end{corollary}
\begin{proof}
Set $(\forall n\in\NN)$ 
$\boldsymbol{\mathsf{R}}_n=(1-\alpha_n^{-1})\ID+
\alpha_n^{-1} \boldsymbol{\mathsf{T}}_n$ and 
$(\forall i\in\{1,\ldots,m\})$ 
$\mathsf{R}_{i,n}=(1-\alpha_n^{-1})\Id+
\alpha_n^{-1}\mathsf{T}_{\!i,n}$.
Moreover, set $(\forall n\in\NN)$ $\mu_n=\alpha_n\lambda_n$ and 
$\boldsymbol{b}_{n}=\alpha_n^{-1}\boldsymbol{a}_{n}$. Then 
$(\forall n\in\NN)$ 
$\Fix\boldsymbol{\mathsf{R}}_n=\Fix\boldsymbol{\mathsf{T}}_{\!n}$ 
and $\boldsymbol{\mathsf{R}}_n$ is nonexpansive. In addition, we 
derive from \eqref{egb18T3-02} that
\begin{equation}
\label{egb18T2-28}
(\forall n\in\NN)(\forall i\in\{1,\ldots,m\})\quad
x_{i,n+1}=x_{i,n}+\varepsilon_{i,n}\mu_n\big({\mathsf R}_{i,n}\,
\boldsymbol{x}_n+b_{i,n}-x_{i,n}\big).
\end{equation}
Since $(\mu_n)_{n\in\NN}$ lies in $[\chi,1-\chi]$ and
$\sum_{n\in\NN}\sqrt{\EC{\|\boldsymbol{b}_n\|^2}{\XXX_n}}=
\sum_{n\in\NN}\alpha_n^{-1}\sqrt{\EC{\|\boldsymbol{a}_n\|^2}{\XXX_n}}
<\pinf$, the result follows from Theorem~\ref{tgb18T2-12} and
Remark~\ref{r:match}\ref{r:matchiii}.
\end{proof}

\begin{remark}
\label{rgb18T3-17}
In the special case of a single-block (i.e., $m=1$) and of 
deterministic errors, Corollary~\ref{cgb18T3-03} reduces to a 
scenario found in \cite[Theorem~4.2]{Opti04}. 
\end{remark}

\section{Double-layer random block-coordinate fixed point 
algorithms}
\label{sec:4}

The algorithm analyzed in this section comprises two successive 
applications of nonexpansive operators at each iteration.
We recall that Notation~\ref{n:2} is in force.

\begin{theorem}
\label{t:1}
Let $(\alpha_n)_{n\in\NN}$ and $(\beta_n)_{n\in\NN}$ be sequences 
in $\zeroun$ such that $\sup_{n\in\NN}\alpha_n<1$ and 
$\sup_{n\in\NN}\beta_n<1$, let $(\lambda_n)_{n\in\NN}$ be a 
sequence in $\left]0,1\right]$ such that
$\inf_{n\in\NN}\lambda_n>0$, and set 
$\mathsf{D}=\{0,1\}^m\smallsetminus\{\boldsymbol{\mathsf{0}}\}$.
Let $\boldsymbol{x}_0$, $(\boldsymbol{a}_n)_{n\in\NN}$, 
and $(\boldsymbol{b}_n)_{n\in\NN}$ be $\HHH$-valued random 
variables, and let $(\boldsymbol{\varepsilon}_n)_{n\in\NN}$ be 
identically distributed $\mathsf{D}$-valued random variables. For 
every $n\in\NN$, let $\boldsymbol{\mathsf R}_n\colon\HHH\to\HHH$ 
be $\beta_n$-averaged and let
$\boldsymbol{\mathsf T}_{\!n}\colon\HHH\to\HHH
\colon\boldsymbol{\mathsf x}\mapsto({\mathsf T}_{\!i,n}
\boldsymbol{\mathsf x})_{1\leq i\leq m}$ be $\alpha_n$-averaged, 
where, $(\forall i\in\{1,\ldots,m\})$ 
${\mathsf T}_{\!i,n}\colon\HHH\to\HH_i$. Iterate
\begin{equation}
\label{e:main1}
\begin{array}{l}
\text{for}\;n=0,1,\ldots\\
\left\lfloor
\begin{array}{l}
\boldsymbol{y}_n=\boldsymbol{\mathsf R}_n\boldsymbol{x}_n+
\boldsymbol{b}_n\\
\text{for}\;i=1,\ldots,m\\
\left\lfloor
\begin{array}{l}
x_{i,n+1}=x_{i,n}+\varepsilon_{i,n}\lambda_n\big({\mathsf T}_{\!i,n}
\boldsymbol{y}_n+a_{i,n}-x_{i,n}\big),
\end{array} 
\right.
\end{array} 
\right.
\end{array} 
\end{equation}
and set $(\forall n\in\NN)$
$\EEE_n=\sigma(\boldsymbol{\varepsilon}_n)$.
In addition, assume that the following hold:
\begin{enumerate}
\item
\label{t:1i}
$\boldsymbol{\mathsf{F}}=\bigcap_{n\in\NN}
\Fix(\boldsymbol{\mathsf T}_{\!n}\circ
\boldsymbol{\mathsf R}_n)\neq\emp$.
\item
\label{t:1ii}
$\sum_{n\in\NN}\sqrt{\EC{\|\boldsymbol{a}_n\|^2}
{\boldsymbol{\XX}_n}}<\pinf$ and
$\sum_{n\in\NN}\sqrt{\EC{\|\boldsymbol{b}_n\|^2}
{\boldsymbol{\XX}_n}}<\pinf$.
\item
\label{t:1iv-}
For every $n\in\NN$, $\EEE_n$ and $\XXX_n$ are independent.
\item
\label{t:1iv}
$(\forall i\in\{1,\ldots,m\})$ 
$\mathsf{p}_i=\PP[\varepsilon_{i,0}=1]>0$.
\end{enumerate}
Then
\begin{equation}
\label{egb18T4-23a}
\big[~(\forall\boldsymbol{\mathsf{z}}\in\boldsymbol{\mathsf{F}})
\;\:\boldsymbol{\mathsf T}_n(\boldsymbol{\mathsf{R}}_n
\boldsymbol{x}_n)-\boldsymbol{\mathsf{R}}_n\boldsymbol{x}_n
+\boldsymbol{\mathsf R}_n\boldsymbol{\mathsf{z}}
\to\boldsymbol{\mathsf{z}}~\big]\;\as
\end{equation}
and
\begin{equation}
\label{egb18T4-23b}
\big[~(\forall\boldsymbol{\mathsf{z}}\in\boldsymbol{\mathsf{F}})
\;\:\boldsymbol{x}_n-\boldsymbol{\mathsf{R}}_n
\boldsymbol{x}_n+\boldsymbol{\mathsf R}_n
\boldsymbol{\mathsf{z}}\to\boldsymbol{\mathsf{z}}~\big]\;\as
\end{equation}
Furthermore, suppose that:
\begin{enumerate}
\setcounter{enumi}{4}
\item
\label{t:1iii}
$\WC(\boldsymbol{x}_n)_{n\in\NN}
\subset\boldsymbol{\mathsf{F}}\;\:\as$
\end{enumerate}
Then $(\boldsymbol{x}_n)_{n\in\NN}$ converges weakly
$\as$ to an $\boldsymbol{\mathsf{F}}$-valued random variable 
$\boldsymbol{x}$. If, in addition, 
\begin{enumerate}
\setcounter{enumi}{5}
\item
\label{t:1x}
$\SC(\boldsymbol{x}_n)_{n\in\NN}\neq\emp\;\:\as$,
\end{enumerate}
then $(\boldsymbol{x}_n)_{n\in\NN}$ converges strongly $\as$ to
$\boldsymbol{x}$.
\end{theorem}
\begin{proof}
Let us prove that the result is an application of 
Theorem~\ref{tHj7+1-12} in the renormed Hilbert space 
$(\HHH,|||\cdot|||)$, where $|||\cdot|||$ is defined in 
\eqref{egb18T2-14}. Note that
\begin{equation}
\label{e:normequiv}
(\forall\boldsymbol{\mathsf{x}}\in\HHH)\quad
\|\boldsymbol{\mathsf{x}}\|^2\leq |||\boldsymbol{\mathsf{x}}|||^2
\leq\frac{1}{\displaystyle{\min_{1\leq i\leq m}}\mathsf{p}_i}
\|\boldsymbol{\mathsf{x}}\|^2
\end{equation}
and that, since the operators $(\boldsymbol{\mathsf{R}}_n\circ
\boldsymbol{\mathsf{T}}_n)_{n\in\NN}$ are nonexpansive, the sets
$(\Fix(\boldsymbol{\mathsf T}_{\!n}\circ
\boldsymbol{\mathsf R}_n))_{n\in\NN}$ are closed 
\cite[Corollary~4.15]{Livre1}, and so is 
$\boldsymbol{\mathsf{F}}$.
Next, for every $n\in\NN$, set 
$\boldsymbol{r}_n=\boldsymbol{\mathsf R}_n\boldsymbol{x}_n$, and 
define $\boldsymbol{t}_n$, $\boldsymbol{c}_n$, $\boldsymbol{d}_n$, 
and $\boldsymbol{e}_n$ coordinate-wise by
\begin{align}
\label{eHj7+2-29a}
(\forall i\in\{1,\ldots,m\})\quad
\begin{cases}
t_{i,n}=x_{i,n}+
\varepsilon_{i,n} ({\mathsf T}_{\!i,n}\boldsymbol{r}_n -x_{i,n})\\
c_{i,n}=\varepsilon_{i,n} a_{i,n} \\
d_{i,n}=\varepsilon_{i,n}({\mathsf T}_{\!i,n}\boldsymbol{y}_n 
-{\mathsf T}_{\!i,n}\boldsymbol{r}_n)
\end{cases}
\text{and}\quad e_{i,n}=c_{i,n}+d_{i,n}.
\end{align}
Then \eqref{e:main1} implies that
\begin{equation}
\label{eHj7+2-28a}
(\forall n\in\NN)\quad
\boldsymbol{x}_{n+1}=\boldsymbol{x}_n+\lambda_n\big(\boldsymbol{t}_n
+\boldsymbol{e}_n-\boldsymbol{x}_n\big).
\end{equation}
On the other hand, we derive from \eqref{eHj7+2-29a} that
\begin{align}\label{e:enanbn}
(\forall n\in\NN)\quad
\sqrt{\EC{|||\boldsymbol{e}_n|||^2}{\boldsymbol{\XX}_n}}
&\leq\sqrt{\EC{|||\boldsymbol{c}_n|||^2}{\boldsymbol{\XX}_n}}+
\sqrt{\EC{|||\boldsymbol{d}_n|||^2}{\boldsymbol{\XX}_n}}\nonumber\\
&\leq\sqrt{\EC{|||\boldsymbol{a}_n|||^2}{\boldsymbol{\XX}_n}}+
\sqrt{\EC{|||\boldsymbol{d}_n|||^2}{\boldsymbol{\XX}_n}}.
\end{align}
However, it follows from \eqref{eHj7+2-29a},
\eqref{e:normequiv}, and the 
nonexpansiveness of the operators
$(\boldsymbol{\mathsf T}_n)_{n\in\NN}$ that
\begin{align}
(\forall n\in\NN)\quad
\EC{|||\boldsymbol{d}_n|||^2}{\boldsymbol{\XX}_n}
&\leq\frac{1}{\displaystyle{\min_{1\leq i\leq m}}\mathsf{p}_i}
\EEC{\sum_{i=1}^m\|\varepsilon_{i,n}({\mathsf T}_{\!i,n}
\boldsymbol{y}_n-{\mathsf T}_{\!i,n}\boldsymbol{r}_n)\|^2}
{\boldsymbol{\XX}_n}\nonumber\\
&\leq\frac{1}{\displaystyle{\min_{1\leq i\leq m}}\mathsf{p}_i}
\EC{\|\boldsymbol{\mathsf{T}}_{\!n}\boldsymbol{y}_n 
-\boldsymbol{\mathsf{T}}_{\!n}\boldsymbol{r}_n\|^2}
{\boldsymbol{\XX}_n}\nonumber\\
&\leq\frac{1}{\displaystyle{\min_{1\leq i\leq m}}\mathsf{p}_i}
\EC{\|\boldsymbol{y}_n-\boldsymbol{r}_n\|^2}{\boldsymbol{\XX}_n}
\nonumber\\
&=\frac{1}{\displaystyle{\min_{1\leq i\leq m}}\mathsf{p}_i}
\EC{\|\boldsymbol{b}_n\|^2}{\boldsymbol{\XX}_n}.
\end{align}
Consequently \eqref{e:normequiv}, \eqref{e:enanbn}, and 
\ref{t:1ii} yield
\begin{align}
\sum_{n\in\NN}\lambda_n\sqrt{\EC{|||\boldsymbol{e}_n|||^2}
{\boldsymbol{\XX}_n}}
&\leq\frac{1}{\displaystyle{\min_{1\leq i\leq m}}
\sqrt{\mathsf{p}_i}}\,
\bigg(\sum_{n\in\NN}\sqrt{\EC{\|\boldsymbol{a}_n\|^2}{
\boldsymbol{\XX}_n}}
+\sum_{n\in\NN}\sqrt{\EC{\|\boldsymbol{b}_n\|^2}{
\boldsymbol{\XX}_n}}\bigg)\nonumber\\
&<\pinf.
\end{align}
Now let $\boldsymbol{\mathsf{z}}\in\boldsymbol{\mathsf{F}}$, and 
set 
\begin{equation}
\label{egb18T4-14a}
(\forall n\in\NN)(\forall i\in\{1,\ldots,m\})\quad
\mathsf{q}_{i,n}\colon\HHH\times\mathsf{D}\to\RR\colon
(\boldsymbol{\mathsf{x}},\boldsymbol{\epsilon})\mapsto
\|\mathsf{x}_i-\mathsf{z}_i+\epsilon_i({\mathsf T}_{\!i,n}
(\boldsymbol{\mathsf{R}}_n\boldsymbol{\mathsf{x}})-
\mathsf{x}_i)\|^2.
\end{equation}
Observe that, for every $n\in\NN$ and every $i\in\{1,\ldots,m\}$,
by continuity of $\boldsymbol{\mathsf{R}}_n$ and 
${\mathsf T}_{\!i,n}$,
${\mathsf T}_{\!i,n}\circ\boldsymbol{\mathsf{R}}_n$ is measurable, 
and the functions $(\mathsf{q}_{i,n}(\cdot,\boldsymbol{\epsilon}))%
_{\boldsymbol{\epsilon}\in\mathsf{D}}$ are therefore likewise. 
Consequently, using \ref{t:1iv} and arguing as in 
\eqref{egb18T4-15c} leads to
\begin{align}
\label{egb18T4-14d}
\hskip -4mm
&(\forall n\in\NN)(\forall i\in\{1,\ldots,m\})\quad
\EC{\|t_{i,n}-\mathsf{z}_{i}\|^2}{\XXX_n}\nonumber\\
&\hskip 62mm =\sum_{\boldsymbol{\epsilon}\in\mathsf{D}}
\EC{\mathsf{q}_{i,n}(\boldsymbol{x}_n,\boldsymbol{\epsilon})
1_{[\boldsymbol{\varepsilon}_n=\boldsymbol{\epsilon}]}}{\XXX_n}
\nonumber\\
&\hskip 62mm =\sum_{\boldsymbol{\epsilon}\in\mathsf{D}}
\PP[\boldsymbol{\varepsilon}_n=\boldsymbol{\epsilon}]
\|x_{i,n}-\mathsf{z}_i+\epsilon_i({\mathsf T}_{\!i,n}
\boldsymbol{r}_n-x_{i,n})\|^2.
\end{align}
Hence, recalling \eqref{egb18T2-14} and \ref{t:1iv}, we obtain
\begin{align}
\label{egb18T4-14k}
&\hskip -4mm
(\forall n\in\NN)\quad
\EC{|||\boldsymbol{t}_n-\boldsymbol{\mathsf{z}}|||^2}
{\boldsymbol{\XX}_n}\nonumber\\
&=\sum_{i=1}^m\frac{1}{\mathsf{p}_i}
\EC{\|t_{i,n}-\mathsf{z}_{i}\|^2}{\XXX_n}\nonumber\\
&=\sum_{i=1}^m\frac{1}{\mathsf{p}_i}
\sum_{\boldsymbol{\epsilon}\in\mathsf{D}}
\PP[\boldsymbol{\varepsilon}_n=\boldsymbol{\epsilon}]
\|x_{i,n}-\mathsf{z}_i+\epsilon_i({\mathsf T}_{\!i,n}
\boldsymbol{r}_n-x_{i,n})\|^2
\nonumber\\
&=\sum_{i=1}^m\frac{1}{\mathsf{p}_i}\left(
\sum_{\boldsymbol{\epsilon}\in\mathsf{D},\epsilon_i=1}
\PP[\boldsymbol{\varepsilon}_n=\boldsymbol{\epsilon}]\,
\|\mathsf{T}_{\!i,n}\boldsymbol{r}_n-\mathsf{z}_{i}\|^2+
\sum_{\boldsymbol{\epsilon}\in\mathsf{D},\epsilon_i=0}
\PP[\boldsymbol{\varepsilon}_n=\boldsymbol{\epsilon}]\,
\|x_{i,n}-\mathsf{z}_{i}\|^2\right)\nonumber\\
&=\|\boldsymbol{\mathsf{T}}_{\!n}\boldsymbol{r}_n-
\boldsymbol{\mathsf{z}}\|^2+
\sum_{i=1}^m \frac{1-\mathsf{p}_i}{\mathsf{p}_i} 
\|x_{i,n}-\mathsf{z}_{i}\|^2\nonumber\\
&=|||\boldsymbol{x}_n-\boldsymbol{\mathsf z}|||^2+
\|\boldsymbol{\mathsf T}_{\!n}\boldsymbol{r}_n-
\boldsymbol{\mathsf z}\|^2
-\|\boldsymbol{x}_n-\boldsymbol{\mathsf z}\|^2.
\end{align}
However, we deduce from \ref{t:1i} and Proposition~\ref{p:av1} that
\begin{equation}
\label{e:avril2014}
(\forall n\in\NN)\quad
\|\boldsymbol{\mathsf T}_{\!n} 
\boldsymbol{r}_n-\boldsymbol{\mathsf z}\|^2 +
\frac{1-\alpha_n}{\alpha_n}
\|\boldsymbol{r}_n-\boldsymbol{\mathsf T}_{\!n} 
\boldsymbol{r}_n-\boldsymbol{\mathsf R}_n\boldsymbol{\mathsf z}
+\boldsymbol{\mathsf z}\|^2
\leq\|\boldsymbol{r}_n-\boldsymbol{\mathsf R}_n
\boldsymbol{\mathsf z}\|^2.
\end{equation}
Combining \eqref{egb18T4-14k} with \eqref{e:avril2014} yields
\begin{multline}
\label{e:2013-2014}
(\forall n\in\NN)\quad
\EC{|||\boldsymbol{t}_n-\boldsymbol{\mathsf z}|||^2}
{\boldsymbol{\XX}_n}+\frac{1-\alpha_n}{\alpha_n} 
\|\boldsymbol{r}_n-\boldsymbol{\mathsf T}_n \boldsymbol{r}_n-
\boldsymbol{\mathsf R}_n\boldsymbol{\mathsf z}
+\boldsymbol{\mathsf z}\|^2\\
\leq\, |||\boldsymbol{x}_n-\boldsymbol{\mathsf z}|||^2
+\|\boldsymbol{\mathsf R}_n
\boldsymbol{x}_n-\boldsymbol{\mathsf R}_n\boldsymbol{\mathsf z}\|^2
-\|\boldsymbol{x}_n-\boldsymbol{\mathsf z}\|^2.
\end{multline}
Now set $\chi=\text{min}\{1/\text{sup}_{k\in\NN}\alpha_k,
1/\text{sup}_{k\in\NN}\beta_k\}-1$. Then $\chi\in\RPP$ and since,
for every $n\in\NN$, $\boldsymbol{\mathsf R}_n$ is 
$\beta_n$-averaged, Proposition~\ref{p:av1} and 
\eqref{e:2013-2014} yield
\begin{equation}
(\forall n\in\NN)\quad
\EC{|||\boldsymbol{t}_n-\boldsymbol{\mathsf z}|||^2}
{\boldsymbol{\XX}_n}+\theta_n(\boldsymbol{\mathsf z})\leq 
|||\boldsymbol{x}_n-\boldsymbol{\mathsf z}|||^2,
\end{equation}
where
\begin{align}
(\forall n\in\NN)\quad
\theta_n(\boldsymbol{\mathsf z})
&=\chi\big(
\|\boldsymbol{r}_n-\boldsymbol{\mathsf T}_{\!n} \boldsymbol{r}_n-
\boldsymbol{\mathsf R}_n\boldsymbol{\mathsf z}+
\boldsymbol{\mathsf z}\|^2+
\|\boldsymbol{x}_n-\boldsymbol{r}_n-\boldsymbol{\mathsf z}
+\boldsymbol{\mathsf R}_n\boldsymbol{\mathsf z}\|^2\big)
\label{egb18T4-20k}\\
&\leq\frac{1-\alpha_n}{\alpha_n} 
\|\boldsymbol{r}_n-\boldsymbol{\mathsf T}_{\!n} \boldsymbol{r}_n-
\boldsymbol{\mathsf R}_n\boldsymbol{\mathsf z}+
\boldsymbol{\mathsf z}\|^2+\frac{1-\beta_n}{\beta_n} 
\|\boldsymbol{x}_n-\boldsymbol{r}_n-\boldsymbol{\mathsf z}
+\boldsymbol{\mathsf R}_n\boldsymbol{\mathsf z}\|^2.
\label{egb18T3-27}
\end{align}
We have thus shown that properties 
\ref{tHj7+1-12ii}--\ref{tHj7+1-12iii} of 
Theorem~\ref{tHj7+1-12} hold with $(\forall n\in\NN)$
$\mu_n=\nu_n=0$. 
Next, let $\boldsymbol{\mathsf{Z}}$ be a countable set which is 
dense in $\boldsymbol{\mathsf{F}}$. Then 
\eqref{egb18T4-19a} asserts that
\begin{equation}
\label{egb18T4-16a}
(\forall\boldsymbol{\mathsf{z}}\in\boldsymbol{\mathsf{Z}})
(\exi\Omega_{\boldsymbol{\mathsf{z}}}\in\FF)\quad
\PP(\Omega_{\boldsymbol{\mathsf{z}}})=1\quad\text{and}\quad
(\forall\omega\in\Omega_{\boldsymbol{\mathsf{z}}})\quad
\sum_{n\in\NN}\lambda_n\theta_n(\boldsymbol{\mathsf{z}},\omega)
<\pinf.
\end{equation}
Moreover, the event 
$\widetilde{\Omega}=\bigcap_{\boldsymbol{\mathsf{z}}\in\mathsf{Z}}
\Omega_{\boldsymbol{\mathsf{z}}}$ is almost certain, i.e., 
$\PP(\widetilde{\Omega})=1$.
Now fix $\boldsymbol{\mathsf{z}}\in\boldsymbol{\mathsf{F}}$.
By density, we can extract from
$\boldsymbol{\mathsf{Z}}$ a sequence 
$(\boldsymbol{\mathsf{z}}_k)_{k\in\NN}$ such that 
$\boldsymbol{\mathsf{z}}_{k}\to\boldsymbol{\mathsf{z}}$. 
In turn, since $\inf_{n\in\NN}\lambda_n>0$, we derive from 
\eqref{egb18T4-20k} and \eqref{egb18T4-16a} that
\begin{equation}
\label{egb18T4-23x}
(\forall k\in\NN)(\forall\omega\in\widetilde{\Omega})\quad
\begin{cases}
\boldsymbol{r}_n(\omega)-\boldsymbol{\mathsf T}_{\!n}
\boldsymbol{r}_n(\omega)-
\boldsymbol{\mathsf R}_n\boldsymbol{\mathsf z}_k+
\boldsymbol{\mathsf z}_k\to\boldsymbol{0}\\
\boldsymbol{x}_n(\omega)-\boldsymbol{r}_n(\omega)-
\boldsymbol{\mathsf z}_k+\boldsymbol{\mathsf R}_n
\boldsymbol{\mathsf z}_k\to\boldsymbol{0}.
\end{cases}
\end{equation}
Now set $\zeta=\sup_{n\in\NN}\sqrt{\beta_n/(1-\beta_n)}$, and 
$(\forall n\in\NN)$
$\boldsymbol{\mathsf{S}}_n=\ID-\boldsymbol{\mathsf{R}}_n$ and 
$\boldsymbol{p}_n=\boldsymbol{r}_n-\boldsymbol{\mathsf T}_{\!n}
\boldsymbol{r}_n$. Then it follows from Proposition~\ref{p:av1} 
that the operators $(\boldsymbol{\mathsf{S}}_n)_{n\in\NN}$ are 
$\zeta$-Lipschitzian. Consequently
\begin{multline}
\label{egb18T4-23o}
(\forall k\in\NN)(\forall n\in\NN)
(\forall\omega\in\widetilde{\Omega})\quad
-\zeta\|\boldsymbol{\mathsf{z}}_k-\boldsymbol{\mathsf{z}}\|
\leq-\|\boldsymbol{\mathsf{S}}_n\boldsymbol{\mathsf{z}}_k-
\boldsymbol{\mathsf{S}}_n\boldsymbol{\mathsf{z}}\|\\
\leq\|\boldsymbol{p}_n(\omega)+
\boldsymbol{\mathsf{S}}_n\boldsymbol{\mathsf{z}}\|-
\|\boldsymbol{p}_n(\omega)+\boldsymbol{\mathsf{S}}_n
\boldsymbol{\mathsf{z}}_k\|
\leq\|\boldsymbol{\mathsf{S}}_n\boldsymbol{\mathsf{z}}_k-
\boldsymbol{\mathsf{S}}_n\boldsymbol{\mathsf{z}}\|
\leq\zeta\|\boldsymbol{\mathsf{z}}_k-\boldsymbol{\mathsf{z}}\|
\end{multline}
and, therefore, \eqref{egb18T4-23x} yields
\begin{align}
\label{egb18T4-23p}
(\forall k\in\NN)\quad
-\zeta\|\boldsymbol{\mathsf{z}}_k-\boldsymbol{\mathsf{z}}\|
&\leq\varliminf_{n\to\pinf}\|\boldsymbol{p}_n(\omega)+
\boldsymbol{\mathsf{S}}_n\boldsymbol{\mathsf{z}}\|-\lim_{n\to\pinf}
\|\boldsymbol{p}_n(\omega)+\boldsymbol{\mathsf{S}}_n
\boldsymbol{\mathsf{z}}_k\|\nonumber\\
&=\varliminf_{n\to\pinf}\|\boldsymbol{p}_n(\omega)+
\boldsymbol{\mathsf{S}}_n\boldsymbol{\mathsf{z}}\|\nonumber\\
&\leq\varlimsup_{n\to\pinf}\|\boldsymbol{p}_n(\omega)+
\boldsymbol{\mathsf{S}}_n\boldsymbol{\mathsf{z}}\|\nonumber\\
&\leq\varlimsup_{n\to\pinf}\|\boldsymbol{p}_n(\omega)+
\boldsymbol{\mathsf{S}}_n\boldsymbol{\mathsf{z}}\|-\lim_{n\to\pinf}
\|\boldsymbol{p}_n(\omega)+\boldsymbol{\mathsf{S}}_n
\boldsymbol{\mathsf{z}}_k\|\nonumber\\
&\leq\zeta\|\boldsymbol{\mathsf{z}}_k-\boldsymbol{\mathsf{z}}\|.
\end{align}
Since $\|\boldsymbol{\mathsf{z}}_k-\boldsymbol{\mathsf{z}}\|\to 0$
and $\PP(\widetilde{\Omega})=1$, we obtain $\boldsymbol{p}_n+
\boldsymbol{\mathsf{S}}_n\boldsymbol{\mathsf{z}}\to\boldsymbol{0}$
$\as$, which proves \eqref{egb18T4-23a}. Likewise, set
$(\forall n\in\NN)$
$\boldsymbol{q}_n=\boldsymbol{x}_n-\boldsymbol{r}_n$. Then, 
proceeding as in \eqref{egb18T4-23p}, \eqref{egb18T4-23x} yields 
$\boldsymbol{q}_n+\boldsymbol{\mathsf{S}}_n
\boldsymbol{\mathsf{z}}\to\boldsymbol{0}$, which
establishes \eqref{egb18T4-23b}. Finally, the weak and strong 
convergence claims follow from \ref{t:1iii}, \ref{t:1x}, and 
Theorem~\ref{tHj7+1-12}.
\end{proof}

\begin{remark}\
\label{rgb18T3-18}
\begin{enumerate}
\item
Consider the special case when only one-block is present ($m=1$)
and when the error sequences $(\boldsymbol{a}_n)_{n\in\NN}$ and 
$(\boldsymbol{b}_n)_{n\in\NN}$,
as well as $\boldsymbol{x}_0$, are deterministic. 
Then the setting of Theorem~\ref{t:1} is found in 
\cite[Theorem~6.3]{Opti04}.
Our framework therefore makes it possible to design
block-coordinate versions of the algorithms which comply with the
two-layer format of \cite[Theorem~6.3]{Opti04}, such as the 
forward-backward algorithm \cite{Opti04} or the algorithms of 
\cite{Luis13} and \cite{Ragu13}. Theorem~\ref{t:1} will be 
applied to block-coordinate forward-backward splitting in 
Section~\ref{sec:52}.
\item
Theorem~\ref{t:1}\ref{t:1iii} gives a condition for the \as\ weak
convergence of a sequence $(\boldsymbol{x}_n)_{n\in\NN}$ produced
by algorithm~\ref{e:main1} to a solution $\boldsymbol{x}$. 
In infinite-dimensional spaces, examples have been constructed for
which the convergence is only weak and not strong, i.e.,  
$(\|\boldsymbol{x}_n-\boldsymbol{x}\|)_{n\in\NN}$ does not 
converge to $0$ \as\ \cite{Smms05,Hund04}. Even if, as in 
Theorem~\ref{t:1}\ref{t:1x}, 
$(\|\boldsymbol{x}_n-\boldsymbol{x}\|)_{n\in\NN}$ does converge to
$0$ \as, there is in general no theoretical upper bound on the 
worst-case behavior of the rate of convergence, which can be
arbitrarily slow \cite{Baus09}. The latter behavior is also 
possible in Euclidean spaces \cite{Baus15,Youl87}.
\end{enumerate}
\end{remark}

\section{Applications to operator splitting}
\label{sec:5}

Let $\mathsf{A}\colon\HH\to 2^{\HH}$ be a set-valued operator and 
let $A^{-1}$ be its inverse, i.e., 
$(\forall(\mathsf{x},\mathsf{u})\in\HH^2)$ 
$\mathsf{x}\in\mathsf{A}^{-1}\mathsf{u}$ $\Leftrightarrow$
$\mathsf{u}\in\mathsf{A}\mathsf{x}$.
The resolvent of $\mathsf{A}$ is 
$\mathsf{J}_\mathsf{A}=(\Id+\mathsf{A})^{-1}$.
The domain of $\mathsf{A}$ is 
$\dom\mathsf{A}=\menge{\mathsf{x}\in\HH}{\mathsf{A}\mathsf{x}
\neq\emp}$ and the graph of $\mathsf{A}$ is 
$\gra\mathsf{A}=
\menge{(\mathsf{x},\mathsf{u})\in\HH\times\HH}{\mathsf{u}\in
\mathsf{A}\mathsf{x}}$.
If $\mathsf{A}$ is monotone, then $\mathsf{J}_\mathsf{A}$ is 
single-valued and nonexpansive and, furthermore, if $\mathsf{A}$ is
maximally monotone, then $\dom\mathsf{J}_\mathsf{A}=\HH$. 
We denote by $\Gamma_0(\HH)$ the class of lower semicontinuous 
convex functions $\mathsf{f}\colon\HH\to\RX$ such that 
$\mathsf{f}\not\equiv\pinf$.
The Moreau subdifferential of $\mathsf{f}\in\Gamma_0(\HH)$ is the 
maximally monotone operator
\begin{equation}
\label{e:subdiff}
\partial\mathsf{f}\colon\HH\to 2^{\HH}\colon\mathsf{x}\mapsto
\menge{\mathsf{u}\in\HH}{(\forall\mathsf{y}\in\HH)\;\;
\scal{\mathsf{y}-\mathsf{x}}{\mathsf{u}}+\mathsf{f}(\mathsf{x})
\leq\mathsf{f}(\mathsf{y})}.
\end{equation}
For every $\mathsf{x}\in\HH$, $\mathsf{f}+\|\mathsf{x}-\cdot\|^2/2$ 
has a unique minimizer, which is denoted by 
$\prox_\mathsf{f}\mathsf{x}$ \cite{Mor62b}. We have 
\begin{equation}
\label{e:prox2}
\prox_{\mathsf{f}}=\mathsf{J}_{\partial\mathsf{f}}.
\end{equation}
For background on convex analysis and monotone 
operator theory, see \cite{Livre1}. We continue to use the
standing Notation~\ref{n:2}.

\subsection{Random block-coordinate Douglas-Rachford splitting}
\label{sec:51}

We propose a random sweeping, block-coordinate version of the 
Douglas-Rachford algorithm with stochastic errors. The purpose 
of this algorithm is to construct iteratively a zero of the sum of 
two maximally monotone operators and it has found applications in 
numerous areas; see, e.g., \cite{Livre1,Bot13c,Luis12,Joca09,%
Banf11,Ecks92,Gaba83,Lion79,Papa14,Penn06,Pesq12}.

\begin{proposition}
\label{pgb18T4-08}
Set $\mathsf{D}=\{0,1\}^m\smallsetminus\{\boldsymbol{\mathsf{0}}\}$ 
and, for every $i\in\{1,\ldots,m\}$, let 
$\mathsf{A}_i\colon\HH_i\to 2^{\HH_i}$ be maximally monotone 
and let $\mathsf{B}_i\colon\HHH\to 2^{\HH_i}$. Suppose that 
$\boldsymbol{\mathsf{B}}\colon\HHH\to 2^{\HHH}\colon\boldsymbol
{\mathsf{x}}\mapsto\cart_{\!i=1}^{\!m}\mathsf{B}_i
\boldsymbol{\mathsf{x}}$ is maximally monotone and that the 
set $\boldsymbol{\mathsf{F}}$ of solutions to the problem
\begin{equation}
\label{egb18T4-08x}
\text{find}\;\;{\mathsf{x}_1\in\HH_1,\ldots,\mathsf{x}_m\in\HH_m}
\;\;\text{such that}\;\;(\forall i\in\{1,\ldots,m\})\quad 0\in
\mathsf{A}_i\mathsf{x}_i+\mathsf{B}_i(\mathsf{x}_1,\ldots,
\mathsf{x}_m)
\end{equation}
is nonempty. Set 
$\boldsymbol{\mathsf{B}}^{-1}\colon\boldsymbol{\mathsf{u}}\mapsto
\cart_{\!i=1}^{\!m}\mathsf{C}_i\boldsymbol{\mathsf{u}}$ where, for 
every $i\in \{1,\ldots,m\}$, $\mathsf{C}_i\colon \HH\to 2^{\HH_i}$.
We also consider the set $\boldsymbol{\mathsf{F}}^*$ of solutions to
the dual problem
\begin{equation}
\label{egb18T4-08y}
\text{find}\;\;{\mathsf{u}_1\in\HH_1,\ldots,\mathsf{u}_m\in\HH_m}
\;\;\text{such that}\;\;(\forall i\in\{1,\ldots,m\})\; 0\in
-\mathsf{A}_i^{-1}(-\mathsf{u}_i)+\mathsf{C}_i(\mathsf{u}_1,\ldots,
\mathsf{u}_m).
\end{equation}
Let $\gamma\in\RPP$, let $(\mu_n)_{n\in\NN}$ be a sequence in 
$\left]0,2\right[$ such that $\inf_{n\in\NN}\mu_n>0$ and 
$\sup_{n\in\NN}\mu_n<2$, let $\boldsymbol{x}_0$, 
$\boldsymbol{z}_0$, $(\boldsymbol{a}_n)_{n\in\NN}$, 
and $(\boldsymbol{b}_n)_{n\in\NN}$ be $\HHH$-valued random 
variables, and let $(\boldsymbol{\varepsilon}_n)_{n\in\NN}$ be 
identically distributed $\mathsf{D}$-valued random variables. Set
$\boldsymbol{\mathsf{J}}_{\gamma\boldsymbol{\mathsf{\mathsf{B}}}}
\colon\boldsymbol{\mathsf{x}}\mapsto
(\mathsf{Q}_i\boldsymbol{\mathsf{x}})_{1\leq i\leq m}$
where, for every $i\in\{1,\ldots,m\}$, 
$\mathsf{Q}_i\colon\HHH\to\HH_i$, iterate
\begin{equation}
\label{e:main15}
\begin{array}{l}
\text{for}\;n=0,1,\ldots\\
\left\lfloor
\begin{array}{l}
\text{for}\;i=1,\ldots,m\\
\left\lfloor
\begin{array}{l}
z_{i,n+1}=z_{i,n}+\varepsilon_{i,n}\big(\mathsf{Q}_i
(x_{1,n},\ldots,x_{m,n})+b_{i,n}-z_{i,n}\big)\\[1mm]
x_{i,n+1}=x_{i,n}+\varepsilon_{i,n}\mu_n
\big(\mathsf{J}_{\gamma\mathsf{A}_i}
(2z_{i,n+1}-x_{i,n})+a_{i,n}-z_{i,n+1}\big),
\end{array}
\right.
\end{array}
\right.\\
\end{array}
\end{equation}
and set $(\forall n\in\NN)$
$\EEE_n=\sigma(\boldsymbol{\varepsilon}_n)$.
Assume that the following hold:
\begin{enumerate}
\item
\label{pgb18T4-08ii}
$\sum_{n\in\NN}\sqrt{\EC{\|\boldsymbol{a}_n\|^2}{
\boldsymbol{\XX}_n}}<\pinf$ and
$\sum_{n\in\NN}\sqrt{\EC{\|\boldsymbol{b}_n\|^2}{
\boldsymbol{\XX}_n}}<\pinf$.
\item
\label{pgb18T4-08iv-}
For every $n\in\NN$, $\EEE_n$ and $\XXX_n$ 
are independent.
\item
\label{pgb18T4-08iv}
$(\forall i\in\{1,\ldots,m\})$ 
$\mathsf{p}_i=\PP[\varepsilon_{i,0}=1]>0$.
\end{enumerate}
Then $(\boldsymbol{x}_n)_{n\in\NN}$ converges weakly $\as$ to a
$\HHH$-valued random variable $\boldsymbol{x}$ such that 
$\boldsymbol{z}=\boldsymbol{\mathsf{J}}_{\gamma
\boldsymbol{\mathsf{B}}}\boldsymbol{x}$ is an 
$\boldsymbol{\mathsf F}$-valued random variable and 
$\boldsymbol{u}=\gamma^{-1}(\boldsymbol{x}-\boldsymbol{z})$ is an 
$\boldsymbol{\mathsf F}^*$-valued random variable.
Furthermore, suppose that:
\begin{enumerate}
\setcounter{enumi}{3}
\item
\label{pgb18T4-08iii}
$\mathsf{J}_{\gamma\boldsymbol{\mathsf{B}}}$ is weakly 
sequentially continuous and 
$\boldsymbol{b}_n\weakly\boldsymbol{0}\;\as$ 
\end{enumerate}
Then $\boldsymbol{z}_n\weakly\boldsymbol{z}\;\:\as$ and 
$\gamma^{-1}(\boldsymbol{x}_n-\boldsymbol{z}_n)\weakly
\boldsymbol{u}\;\:\as$
\end{proposition}
\begin{proof}
Set $\boldsymbol{\mathsf{A}}\colon\HHH\to 2^{\HHH}\colon
\boldsymbol{\mathsf{x}}\mapsto\cart_{\!i=1}^{\!m}
\mathsf{A}_i\mathsf{x}_i$ and $(\forall i\in\{1,\ldots,m\})$ 
$\mathsf{T}_{\!i}=(2\mathsf{J}_{\gamma\mathsf{A}_i}-\Id)\circ
(2\mathsf{Q}_i-\Id)$. 
Then $\boldsymbol{\mathsf{T}}=(2\boldsymbol{\mathsf{J}}_{\gamma%
\boldsymbol{\mathsf{A}}}-\ID)\circ
(2\boldsymbol{\mathsf{J}}_{\gamma\boldsymbol{\mathsf{B}}}-\ID)$ 
is nonexpansive as the composition of two nonexpansive operators
\cite[Corollary~23.10(ii)]{Livre1}. Furthermore
$\Fix\boldsymbol{\mathsf{T}}\neq\emp$ since
\cite[Lemma~2.6(iii)]{Opti04}
\begin{equation}
\label{egb18T4-27a}
\boldsymbol{\mathsf{J}}_{\gamma\boldsymbol{\mathsf{B}}}
(\Fix\boldsymbol{\mathsf{T}})=\zer(\boldsymbol{\mathsf{A}}+
\boldsymbol{\mathsf{B}})=\boldsymbol{\mathsf{F}}\neq\emp. 
\end{equation}
Now set
\begin{equation}
\label{egb18T4-08j}
(\forall n\in\NN)\quad\lambda_n=\mu_n/2\quad\text{and}\quad
\boldsymbol{e}_n=
2\big(\boldsymbol{\mathsf{J}}_{\gamma\boldsymbol{\mathsf{A}}}
(2\boldsymbol{\mathsf{J}}_{\gamma\boldsymbol{\mathsf{B}}}
\boldsymbol{x}_n+2\boldsymbol{b}_n-\boldsymbol{x}_n)-
\boldsymbol{\mathsf{J}}_{\gamma\boldsymbol{\mathsf{A}}}
(2\boldsymbol{\mathsf{J}}_{\gamma\boldsymbol{\mathsf{B}}}
\boldsymbol{x}_n-\boldsymbol{x}_n)+\boldsymbol{a}_n-
\boldsymbol{b}_n\big).
\end{equation}
Then we derive from \eqref{e:main15} that
\begin{align}
\label{egb18T4-08a}
(\forall n\in\NN)(\forall i\in\{1,\ldots,m\})\quad 
{x}_{i,n+1}
&={x}_{i,n}+{\varepsilon}_{i,n}\mu_n\big(
{\mathsf{J}}_{\gamma{\mathsf{A}_i}}\big(2\mathsf{Q}_i
\boldsymbol{x}_n+
2b_{i,n}-{x}_{i,n}\big)+a_{i,n}-z_{i,n+1}\big)\nonumber\\
&={x}_{i,n}+{\varepsilon}_{i,n}\lambda_n
\big(2{\mathsf{J}}_{\gamma{\mathsf{A}_i}}
\big(2\mathsf{Q}_i\boldsymbol{x}_n-x_{i,n}\big)
+e_{i,n}-2\mathsf{Q}_i\boldsymbol{x}_n\big)\nonumber\\
&=x_{i,n}+{\varepsilon}_{i,n}\lambda_n
\big(\mathsf{T}_{\!i}\boldsymbol{x}_n+e_{i,n}-x_{i,n}\big),
\end{align}
which is precisely the iteration process \eqref{egb18T3-05a}.
Furthermore, we infer from \eqref{egb18T4-08j} and the 
nonexpansiveness of 
$\boldsymbol{\mathsf{J}}_{\gamma\boldsymbol{\mathsf{A}}}$ 
\cite[Corollary~23.10(i)]{Livre1} that
\begin{align}
\label{egb18T4-09a}
(\forall n\in\NN)\quad
\|\boldsymbol{e}_n\|^2
&\leq 4\|\boldsymbol{\mathsf{J}}_{\gamma\boldsymbol{\mathsf{A}}}
(2\boldsymbol{\mathsf{J}}_{\gamma\boldsymbol{\mathsf{B}}}
\boldsymbol{x}_n+2\boldsymbol{b}_n-\boldsymbol{x}_n)-
\boldsymbol{\mathsf{J}}_{\gamma\boldsymbol{\mathsf{A}}}
(2\boldsymbol{\mathsf{J}}_{\gamma\boldsymbol{\mathsf{B}}}
\boldsymbol{x}_n-\boldsymbol{x}_n)+\boldsymbol{a}_n-
\boldsymbol{b}_n\|^2\nonumber\\
&\leq12\big(\|\boldsymbol{\mathsf{J}}_{\gamma\boldsymbol{\mathsf{A}}}
(2\boldsymbol{\mathsf{J}}_{\gamma\boldsymbol{\mathsf{B}}}
\boldsymbol{x}_n+2\boldsymbol{b}_n-\boldsymbol{x}_n)-
\boldsymbol{\mathsf{J}}_{\gamma\boldsymbol{\mathsf{A}}}
(2\boldsymbol{\mathsf{J}}_{\gamma\boldsymbol{\mathsf{B}}}
\boldsymbol{x}_n-\boldsymbol{x}_n)\|^2+\|\boldsymbol{a}_n\|^2
+\|\boldsymbol{b}_n\|^2\big)\nonumber\\
&\leq12\big(\|\boldsymbol{a}_n\|^2+5\|\boldsymbol{b}_n\|^2\big)
\end{align}
and therefore that 
\begin{equation}
(\forall n\in\NN)\quad
\sqrt{\EC{\|\boldsymbol{e}_n\|^2}{\XXX_n}}\leq2\sqrt{3}
\Big(\sqrt{\EC{\|\boldsymbol{a}_n\|^2}{\XXX_n}}+
\sqrt{5}\sqrt{\EC{\|\boldsymbol{b}_n\|^2}{\XXX_n}}\,\Big).
\end{equation}
Thus, we deduce from \ref{pgb18T4-08ii} that
$\sum_{n\in\NN}\sqrt{\EC{\|\boldsymbol{e}_n\|^2}{\XXX_n}}<\pinf$. 
Altogether, the almost sure weak convergence of 
$(\boldsymbol{x}_n)_{n\in\NN}$ to
a ($\Fix\boldsymbol{\mathsf{T}}$)-valued random variable
$\boldsymbol{x}$ follows 
from Corollary~\ref{cgb18T3-04}. In turn, \eqref{egb18T4-27a}
asserts that $\boldsymbol{z}=
\boldsymbol{\mathsf{J}}_{\gamma\boldsymbol{\mathsf{B}}}
\boldsymbol{x}\in\boldsymbol{\mathsf{F}}$ $\as$ 
Now set 
$\boldsymbol{{u}}=\gamma^{-1}(\boldsymbol{{x}}-
\boldsymbol{{z}})$. Then, $\as$,
\begin{equation}
\label{e:firth1}
\boldsymbol{z} =
\boldsymbol{\mathsf{J}}_{\gamma\boldsymbol{\mathsf{B}}}
\boldsymbol{x}\;\Leftrightarrow\;
\boldsymbol{x}-\boldsymbol{{z}} \in 
\gamma\boldsymbol{\mathsf{B}}\boldsymbol{{z}}
\;\Leftrightarrow\;\boldsymbol{{z}}\in
\boldsymbol{\mathsf{B}}^{-1}\boldsymbol{{u}}
\end{equation}
and
\begin{eqnarray}
\label{e:firth2}
\boldsymbol{x}\in\Fix\boldsymbol{\mathsf{T}} 
&\Leftrightarrow& 
\boldsymbol{{x}}=
(2\boldsymbol{\mathsf{J}}_{\gamma\boldsymbol{\mathsf{A}}}-\ID)
(2\boldsymbol{{z}}-\boldsymbol{x})
\nonumber\\
&\Leftrightarrow&
\boldsymbol{{z}}=
\boldsymbol{\mathsf{J}}_{\gamma\boldsymbol{\mathsf{A}}}
(2\boldsymbol{{z}}-\boldsymbol{{x}})
\nonumber\\
&\Leftrightarrow& \boldsymbol{{z}}-\boldsymbol{{x}}
\in\gamma\boldsymbol{\mathsf{A}}\boldsymbol{{z}}
\nonumber\\
&\Leftrightarrow&-\boldsymbol{{z}}\in-
\boldsymbol{\mathsf{A}}^{-1}(-\boldsymbol{{u}}).
\end{eqnarray}
These imply that 
$\boldsymbol{\mathsf{0}}\in -\boldsymbol{\mathsf{A}}^{-1}
(-\boldsymbol{u})+\boldsymbol{\mathsf{B}}^{-1}\boldsymbol{u}$
$\as$, i.e., that $\boldsymbol{u}\in\boldsymbol{\mathsf{F}}^*$ 
$\as$ 
Finally, assume that \ref{pgb18T4-08iii} holds. Then there exists 
$\widetilde{\Omega}\in\FF$ such that $\PP(\widetilde{\Omega})=1$ 
and $(\forall\omega\in\widetilde{\Omega})$ 
$\boldsymbol{\mathsf{J}}_{\gamma\boldsymbol{\mathsf{B}}}
\boldsymbol{x}_n(\omega)\weakly\boldsymbol{\mathsf{J}}_{\gamma
\boldsymbol{\mathsf{B}}}\boldsymbol{x}(\omega)=
\boldsymbol{z}(\omega)$. Now let
$i\in\{1,\ldots,m\}$, $\omega\in\widetilde{\Omega}$, and
$\boldsymbol{\mathsf{v}}\in\HHH$. Then 
$\scal{\mathsf{Q}_i\boldsymbol{x}_n(\omega)}
{\mathsf{v}_i}\to\scal{z_i(\omega)}{\mathsf{v}_i}$ and
\eqref{e:main15} yields
\begin{equation}
(\forall n\in\NN)\quad\scal{z_{i,n+1}(\omega)}{\mathsf{v}_i}=
\scal{z_{i,n}(\omega)}{\mathsf{v}_i}+\varepsilon_{i,n}(\omega)
\big(\scal{\mathsf{Q}_i\boldsymbol{x}_n(\omega)}
{\mathsf{v}_i}+\scal{b_{i,n}(\omega)}{\mathsf{v}_i}-
\scal{z_{i,n}(\omega)}{\mathsf{v}_i}\big).
\end{equation}
However, according to \ref{pgb18T4-08iv}, at the expense of
possibly taking $\omega$ in a smaller almost sure event, 
$\varepsilon_{i,n}(\omega)=1$ infinitely often. Hence, there 
exists a monotone sequence $(k_n)_{n\in\NN}$ in $\NN$ 
such that $k_n\to\pinf$ and, for $n\in\NN$ sufficiently large,
\begin{equation}
\scal{z_{i,n+1}(\omega)}{\mathsf{v}_i}=
\scal{\mathsf{Q}_i\boldsymbol{x}_{k_n}(\omega)}
{\mathsf{v}_i}+\scal{b_{i,k_n}(\omega)}{\mathsf{v}_i}.
\end{equation}
Thus, since $\scal{\mathsf{Q}_i\boldsymbol{x}_{k_n}(\omega)}
{\mathsf{v}_i}\to\scal{z_i(\omega)}{\mathsf{v}_i}$ and
$\scal{b_{i,k_n}(\omega)}{\mathsf{v}_i}\to 0$, 
$\scal{z_{i,n+1}(\omega)-z_i(\omega)}{\mathsf{v}_i}\to 0$.
Hence, $\scal{\boldsymbol{z}_{n+1}(\omega)-
\boldsymbol{z}(\omega)}{\boldsymbol{\mathsf{v}}}=\sum_{i=1}^m
\scal{z_{i,n+1}(\omega)-z_i(\omega)}{\mathsf{v}_i}\to 0$.
This shows that $\boldsymbol{z}_n\weakly\boldsymbol{z}\;\as$,
which allows us to conclude that $\gamma^{-1}(\boldsymbol{x}_n-
\boldsymbol{z}_n)\weakly\boldsymbol{u}\;\as$
\end{proof}

\begin{remark}
\label{rgb18T4-08}
Let us make some connections between
Proposition~\ref{pgb18T4-08} and existing results.
\begin{enumerate}
\item
In the standard case of a single block ($m=1$) and when all the
variables are deterministic, the above primal convergence 
result goes back to \cite{Ecks92} and to \cite{Lion79} in the 
unrelaxed case. 
\item
In minimization problems, the alternating direction method of 
multipliers (ADMM) is strongly related to an application of the 
Douglas-Rachford algorithm to the dual problem \cite{Gaba83}. 
This connection can be used to construct a random block-coordinate 
ADMM algorithm. Let us note that such an algorithm was recently
proposed in \cite{Iutz13} in a finite-dimensional setting, where
single-block, unrelaxed, and error-free iterations were used.
\end{enumerate}
\end{remark}

Next, we apply Proposition~\ref{pgb18T4-08} to devise a 
primal-dual block-coordinate algorithm for solving a class of 
structured inclusion problems investigated in \cite{Siop13}. 

\begin{corollary}
\label{cgb18T4-09}
Set $\mathsf{D}=\{0,1\}^{m+p}\smallsetminus
\{\boldsymbol{\mathsf{0}}\}$,
let $(\GG_k)_{1\leq k\leq p}$ be separable real Hilbert spaces, and 
set $\GGG=\GG_1\oplus\cdots\oplus\GG_p$. For every 
$i\in\{1,\ldots,m\}$, let 
$\mathsf{A}_i\colon\HH_i\to 2^{\HH_i}$ be maximally monotone
and, for every $k\in\{1,\ldots,p\}$, let 
$\mathsf{B}_k\colon\GG_k\to 2^{\GG_k}$ be maximally monotone, and 
let $\mathsf{L}_{ki}\colon\HH_i\to\GG_k$ be linear and bounded. 
It is assumed that the set $\boldsymbol{\mathsf{F}}$ of 
solutions to the problem
\begin{equation}
\label{egb18T4-09p}
\text{find}\;\;{\mathsf{x}_1\in\HH_1,\ldots,\mathsf{x}_m\in\HH_m}
\;\;\text{such that}\;\;(\forall i\in\{1,\ldots,m\})\; 0\in
\mathsf{A}_i\mathsf{x}_i+\sum_{k=1}^p\mathsf{L}_{ki}^*\mathsf{B}_k
\bigg(\sum_{j=1}^m\mathsf{L}_{kj}\mathsf{x}_j\bigg)
\end{equation}
is nonempty. We also consider the set $\boldsymbol{\mathsf{F}}^*$ of 
solutions to the dual problem
\begin{equation}
\label{egb18T4-09d}
\text{find}\;\;{\mathsf{v}_1\in\GG_1,\ldots,\mathsf{v}_p\in\GG_p}
\;\;\text{such that}\;\;(\forall k\in\{1,\ldots,p\})\; 0\in
-\Sum_{i=1}^m\mathsf{L}_{ki}\mathsf{A}_i^{-1}\bigg(-\Sum_{l=1}^p
\mathsf{L}_{li}^*\mathsf{v}_l\bigg)+\mathsf{B}_k^{-1}\mathsf{v}_k.
\end{equation}
Let $\gamma\in\RPP$, let $(\mu_n)_{n\in\NN}$ be a sequence in 
$\left]0,2\right[$ such that $\inf_{n\in\NN}\mu_n>0$ and 
$\sup_{n\in\NN}\mu_n<2$, let $\boldsymbol{x}_0$, $\boldsymbol{z}_0$,
$(\boldsymbol{a}_n)_{n\in\NN}$, and 
$(\boldsymbol{c}_n)_{n\in\NN}$ be $\HHH$-valued random variables,
let $\boldsymbol{y}_0$, $\boldsymbol{w}_0$,
$(\boldsymbol{b}_n)_{n\in\NN}$, and 
$(\boldsymbol{d}_n)_{n\in\NN}$ be $\GGG$-valued random 
variables, and let $(\boldsymbol{\varepsilon}_n)_{n\in\NN}$ be 
identically distributed $\mathsf{D}$-valued random variables. Set
\begin{equation}
\label{egb18T4-10a}
\boldsymbol{\mathsf{V}}=\Menge{(\mathsf{x}_1,\ldots,\mathsf{x}_m,
\mathsf{y}_1,\ldots,\mathsf{y}_p)\in\HHH\oplus\GGG}
{(\forall k\in\{1,\ldots,p\})\;\mathsf{y}_k=
\sum_{i=1}^m\mathsf{L}_{ki}\mathsf{x}_i},
\end{equation}
let $\boldsymbol{\mathsf{P}}_{\boldsymbol{\mathsf{\mathsf{V}}}}
\colon\boldsymbol{\mathsf{x}}\mapsto
(\mathsf{Q}_j\boldsymbol{\mathsf{x}})_{1\leq j\leq m+p}$ be its
projection operator, 
where $(\forall i\in\{1,\ldots,m\})$ $\mathsf{Q}_i \colon
\HHH\oplus \GGG \to \HH_i$
and $(\forall k\in\{1,\ldots,p\})$ $\mathsf{Q}_{m+k}\colon
\HHH\oplus\GGG\to\GG_k$, iterate
\begin{equation}
\label{e:main17}
\begin{array}{l}
\text{for}\;n=0,1,\ldots\\
\left\lfloor
\begin{array}{l}
\text{for}\;i=1,\ldots,m\\
\left\lfloor
\begin{array}{l}
z_{i,n+1}=z_{i,n}+\varepsilon_{i,n}\big(\mathsf{Q}_i(x_{1,n},%
\ldots,x_{m,n},y_{1,n},\ldots,y_{p,n})+c_{i,n}-z_{i,n}\big)\\[1mm]
x_{i,n+1}=x_{i,n}+\varepsilon_{i,n}\mu_n
\big(\mathsf{J}_{\gamma\mathsf{A}_i}
(2z_{i,n+1}-x_{i,n})+a_{i,n}-z_{i,n+1}\big)
\end{array}
\right.\\
\text{for}\;k=1,\ldots,p\\
\left\lfloor
\begin{array}{l}
w_{k,n+1}=w_{k,n}+\varepsilon_{m+k,n}\big(\mathsf{Q}_{m+k}
(x_{1,n},\ldots,x_{m,n},y_{1,n},\ldots,y_{p,n})
+d_{k,n}-w_{k,n}\big)\\[1mm]
y_{k,n+1}=y_{k,n}+\varepsilon_{m+k,n}\mu_n
\big(\mathsf{J}_{\gamma\mathsf{B}_k}
(2w_{k,n+1}-y_{k,n})+b_{k,n}-w_{k,n+1}\big),
\end{array}
\right.
\end{array}
\right.\\
\end{array}
\end{equation}
and set $(\forall n\in\NN)$ $\boldsymbol{\YY}_n=
\sigma(\boldsymbol{x}_j,\boldsymbol{y}_j)_{0\leq j\leq n}$ and
$\EEE_n=\sigma(\boldsymbol{\varepsilon}_n)$.
In addition, assume that the following hold:
\begin{enumerate}
\item
\label{cgb18T4-09ii}
$\sum_{n\in\NN}\sqrt{\EC{\|\boldsymbol{a}_n\|^2}
{\boldsymbol{\YY}_n}}<\pinf$,
$\sum_{n\in\NN}\sqrt{\EC{\|\boldsymbol{b}_n\|^2}
{\boldsymbol{\YY}_n}}<\pinf$,
$\sum_{n\in\NN}\sqrt{\EC{\|\boldsymbol{c}_n\|^2}
{\boldsymbol{\YY}_n}}<\pinf$, 
$\sum_{n\in\NN}\sqrt{\EC{\|\boldsymbol{d}_n\|^2}
{\boldsymbol{\YY}_n}}<\pinf$, 
$\boldsymbol{c}_n\weakly\boldsymbol{\mathsf{0}}\:\as$, and 
$\boldsymbol{d}_n\weakly\boldsymbol{\mathsf{0}}\:\as$
\item
\label{cgb18T4-09iiv-}
For every $n\in\NN$, $\EEE_n$ and $\boldsymbol{\YY}_n$ are 
independent.
\item
\label{cgb18T4-09iiv}
$(\forall j\in\{1,\ldots,m+p\})$ 
$\PP[\varepsilon_{j,0}=1]>0$.
\end{enumerate}
Then $(\boldsymbol{z}_n)_{n\in\NN}$ converges weakly $\as$ to an 
$\boldsymbol{\mathsf{F}}$-valued random variable, and
$(\gamma^{-1}(\boldsymbol{w}_n-\boldsymbol{y}_n))_{n\in\NN}$ 
converges weakly $\as$ to an
$\boldsymbol{\mathsf{F}}^*$-valued random variable.
\end{corollary}
\begin{proof}
Set 
$\boldsymbol{\mathsf{A}}\colon\HHH\to2^{\HHH}\colon\boldsymbol
{\mathsf{x}}\mapsto\cart_{\!i=1}^{\!m}\mathsf{A}_i\mathsf{x}_i$,
$\boldsymbol{\mathsf{B}}\colon\GGG\to2^{\GGG}\colon\boldsymbol
{\mathsf{y}}\mapsto\cart_{\!k=1}^{\!p}\mathsf{B}_k\mathsf{y}_k$,
and $\boldsymbol{\mathsf{L}}\colon\HHH\to\GGG\colon
\boldsymbol{\mathsf{x}}\mapsto\big(\sum_{i=1}^m\mathsf{L}_{ki}
\mathsf{x}_i\big)_{1\leq k\leq p}$.
Furthermore, let us introduce 
\begin{equation}
\label{egb18T4-09r}
\KKK=\HHH\oplus\GGG,\quad
\boldsymbol{\mathsf{C}}\colon\KKK\to 2^{\KKK}
\colon (\boldsymbol{\mathsf{x}},\boldsymbol{\mathsf{y}})\mapsto
\boldsymbol{\mathsf{A}}\boldsymbol{\mathsf{x}}\times
\boldsymbol{\mathsf{B}}\boldsymbol{\mathsf{y}},
\quad\text{and}\quad
\boldsymbol{\mathsf{V}}=\menge{(\boldsymbol{\mathsf{x}},%
\boldsymbol{\mathsf{y}})\in\KKK}{\boldsymbol{\mathsf{L}}
\boldsymbol{\mathsf{x}}=\boldsymbol{\mathsf{y}}}.
\end{equation}
Then the primal-dual problem
\eqref{egb18T4-09p}--\eqref{egb18T4-09d} can be rewritten as
\begin{equation}
\label{egb18T4-09q}
\text{find}\;\;(\boldsymbol{\mathsf{x}},\boldsymbol{\mathsf{v}})
\in\KKK\;\;\text{such that}\;\;
\begin{cases}
\boldsymbol{\mathsf{0}}
\in\boldsymbol{\mathsf{A}}\boldsymbol{\mathsf{x}}+
\boldsymbol{\mathsf{L}}^*\boldsymbol{\mathsf{B}}
\boldsymbol{\mathsf{L}}\boldsymbol{\mathsf{x}}\\
\boldsymbol{\mathsf{0}}\in -\boldsymbol{\mathsf{L}}
\boldsymbol{\mathsf{A}}^{-1}(-\boldsymbol{\mathsf{L}}^*
\boldsymbol{\mathsf{v}})+\boldsymbol{\mathsf{B}}^{-1}
\boldsymbol{\mathsf{v}}.
\end{cases}
\end{equation}
The normal cone operator to $\boldsymbol{\mathsf{V}}$ is
\cite[Example~6.42]{Livre1}
\begin{equation}
\label{egb18T4-09ss}
\boldsymbol{\mathsf{N}}_{\boldsymbol{\mathsf{V}}}
\colon\KKK\to 2^{\KKK}
\colon(\boldsymbol{\mathsf{x}},\boldsymbol{\mathsf{y}})\mapsto
\begin{cases}
\boldsymbol{\mathsf{V}}^\bot,&\text{if}\;\;\boldsymbol{\mathsf{L}}
\boldsymbol{\mathsf{x}}=\boldsymbol{\mathsf{y}};\\
\emp,&\text{if}\;\;\boldsymbol{\mathsf{L}}\boldsymbol{\mathsf{x}}
\neq\boldsymbol{\mathsf{y}},
\end{cases}
\quad\text{where}\quad
\boldsymbol{\mathsf{V}}^\bot=\menge{(\boldsymbol{\mathsf{u}},%
\boldsymbol{\mathsf{v}})\in\KKK}{\boldsymbol{\mathsf{u}}=
-\boldsymbol{\mathsf{L}}^*\boldsymbol{\mathsf{v}}}.
\end{equation}
Now let $(\boldsymbol{\mathsf{x}},\boldsymbol{\mathsf{y}})\in\KKK$.
Then 
\begin{eqnarray}
\label{egb18T4-09s}
(\boldsymbol{\mathsf{0}},\boldsymbol{\mathsf{0}})\in
\boldsymbol{\mathsf{C}}(\boldsymbol{\mathsf{x}},%
\boldsymbol{\mathsf{y}})+
\boldsymbol{\mathsf{N}}_{\boldsymbol{\mathsf{V}}}
(\boldsymbol{\mathsf{x}},\boldsymbol{\mathsf{y}})
&\Leftrightarrow&
\begin{cases}
(\boldsymbol{\mathsf{x}},\boldsymbol{\mathsf{y}})\in\
\boldsymbol{\mathsf{V}}\\
(\boldsymbol{\mathsf{0}},\boldsymbol{\mathsf{0}})\in
(\boldsymbol{\mathsf{A}}\boldsymbol{\mathsf{x}}\times 
\boldsymbol{\mathsf{B}}\boldsymbol{\mathsf{y}})+
\boldsymbol{\mathsf{V}}^\bot
\end{cases}
\nonumber\\
&\Leftrightarrow&
\begin{cases}
\boldsymbol{\mathsf{L}}\boldsymbol{\mathsf{x}}=
\boldsymbol{\mathsf{y}}\\
(\exi\boldsymbol{\mathsf{u}}\in
\boldsymbol{\mathsf{A}}\boldsymbol{\mathsf{x}})
(\exi\boldsymbol{\mathsf{v}}\in
\boldsymbol{\mathsf{B}}\boldsymbol{\mathsf{y}})\;\:
\boldsymbol{\mathsf{u}}=-\boldsymbol{\mathsf{L}}^*
\boldsymbol{\mathsf{v}}
\end{cases}
\nonumber\\
&\Rightarrow&
(\exi\boldsymbol{\mathsf{v}}\in\boldsymbol{\mathsf{B}}
(\boldsymbol{\mathsf{L}}\boldsymbol{\mathsf{x}}))\;
-\boldsymbol{\mathsf{L}}^*
\boldsymbol{\mathsf{v}}\in\boldsymbol{\mathsf{A}}
\boldsymbol{\mathsf{x}}
\nonumber\\
&\Rightarrow&
(\exi\boldsymbol{\mathsf{v}}\in\GGG)\;\:
\boldsymbol{\mathsf{L}}^*\boldsymbol{\mathsf{v}}\in
\boldsymbol{\mathsf{L}}^*\boldsymbol{\mathsf{B}}
\boldsymbol{\mathsf{L}}\boldsymbol{\mathsf{x}}\;\;\text{and}\;
-\boldsymbol{\mathsf{L}}^*
\boldsymbol{\mathsf{v}}\in\boldsymbol{\mathsf{A}}
\boldsymbol{\mathsf{x}}
\nonumber\\
&\Leftrightarrow&
\boldsymbol{\mathsf{x}}\;\text{solves \eqref{egb18T4-09p}}.
\end{eqnarray}
Since $\boldsymbol{\mathsf{C}}$ and 
$\boldsymbol{\mathsf{N}}_{\boldsymbol{\mathsf{V}}}$ are maximally 
monotone, it follows from \cite[Proposition~23.16]{Livre1}
that the iteration process \eqref{e:main17} 
is an instance of \eqref{e:main15} for finding a zero of 
$\boldsymbol{\mathsf{C}}+
\boldsymbol{\mathsf{N}}_{\boldsymbol{\mathsf{V}}}$ in $\KKK$.
The associated dual problem consists of finding a zero of
$-\boldsymbol{\mathsf{C}}^{-1}(-\cdot)+
\boldsymbol{\mathsf{N}}_{\boldsymbol{\mathsf{V}}}^{-1}$. 
Let $(\boldsymbol{\mathsf{u}},\boldsymbol{\mathsf{v}})\in\KKK$.
Then \eqref{egb18T4-09ss} yields
\begin{eqnarray}
\label{e:pasdormi11}
(\boldsymbol{\mathsf{0}},\boldsymbol{\mathsf{0}})\in
-\boldsymbol{\mathsf{C}}^{-1}(-\boldsymbol{\mathsf{u}},%
-\boldsymbol{\mathsf{v}})+
\boldsymbol{\mathsf{N}}_{\boldsymbol{\mathsf{V}}}^{-1}
(\boldsymbol{\mathsf{u}},\boldsymbol{\mathsf{v}})
&\Leftrightarrow&
(\boldsymbol{\mathsf{0}},\boldsymbol{\mathsf{0}})\in
-\boldsymbol{\mathsf{C}}^{-1}(-\boldsymbol{\mathsf{u}},%
-\boldsymbol{\mathsf{v}})+
\boldsymbol{\mathsf{N}}_{\boldsymbol{\mathsf{V}}^\bot}
(\boldsymbol{\mathsf{u}},\boldsymbol{\mathsf{v}})\nonumber\\
&\Leftrightarrow&
\begin{cases}
(\boldsymbol{\mathsf{u}},\boldsymbol{\mathsf{v}})\in\
\boldsymbol{\mathsf{V}}^\bot\\
(\boldsymbol{\mathsf{0}},\boldsymbol{\mathsf{0}})\in
\big(-\boldsymbol{\mathsf{A}}^{-1}(-\boldsymbol{\mathsf{u}})\times 
-\boldsymbol{\mathsf{B}}^{-1}(-\boldsymbol{\mathsf{v}})\big)+
\boldsymbol{\mathsf{V}}
\end{cases}
\nonumber\\
&\Leftrightarrow&
\begin{cases}
\boldsymbol{\mathsf{u}}=-\boldsymbol{\mathsf{L}}^*
\boldsymbol{\mathsf{v}}\\
(\exi\boldsymbol{\mathsf{x}}\in
-\boldsymbol{\mathsf{A}}^{-1}(-\boldsymbol{\mathsf{u}}))
(\exi\boldsymbol{\mathsf{y}}\in
-\boldsymbol{\mathsf{B}}^{-1}(-\boldsymbol{\mathsf{v}}))\;
\boldsymbol{\mathsf{L}}\boldsymbol{\mathsf{x}}=
\boldsymbol{\mathsf{y}}
\end{cases}
\nonumber\\
&\Rightarrow&
(\exi\boldsymbol{\mathsf{x}}\in-\boldsymbol{\mathsf{A}}^{-1}
(\boldsymbol{\mathsf{L}}^*\boldsymbol{\mathsf{v}}))\;
\boldsymbol{\mathsf{L}}
\boldsymbol{\mathsf{x}}\in-\boldsymbol{\mathsf{B}}^{-1}
(-\boldsymbol{\mathsf{v}})
\nonumber\\
&\Rightarrow&
(\exi\boldsymbol{\mathsf{x}}\in\HHH)\;\;
\boldsymbol{\mathsf{L}}\boldsymbol{\mathsf{x}}\in
-\boldsymbol{\mathsf{L}}\boldsymbol{\mathsf{A}}^{-1}
(\boldsymbol{\mathsf{L}}^*\boldsymbol{\mathsf{v}}))
\;\text{and}\;
-\boldsymbol{\mathsf{L}}
\boldsymbol{\mathsf{x}}\in\boldsymbol{\mathsf{B}}^{-1}
(-\boldsymbol{\mathsf{v}})
\nonumber\\
&\Leftrightarrow&
-\boldsymbol{\mathsf{v}}\;\text{solves \eqref{egb18T4-09d}}.
\end{eqnarray}
The convergence result therefore follows from 
Proposition~\ref{pgb18T4-08} using \eqref{egb18T4-09s}, 
\eqref{e:pasdormi11}, and the weak continuity of 
$\boldsymbol{\mathsf{P}}_{\boldsymbol{\mathsf{\mathsf{V}}}}=
\boldsymbol{\mathsf{J}}_{\gamma \boldsymbol{\mathsf{N}}_
{\boldsymbol{\mathsf{\mathsf{V}}}}}$
\cite[Proposition~28.11(i)]{Livre1}.
\end{proof}

\begin{remark}
The parametrization \eqref{egb18T4-09r} made it possible to
reduce the structured primal-dual problem 
\eqref{egb18T4-09p}--\eqref{egb18T4-09d} to a basic
two-operator inclusion, to which the block-coordinate
Douglas-Rachford algorithm \eqref{e:main15} could be applied. 
A similar parametrization was used in \cite{Optl14} in a different
context. We also note that, at each iteration of 
Algorithm~\eqref{e:main17}, components of the projector
$\boldsymbol{\mathsf{P}}_{\boldsymbol{\mathsf{\mathsf{V}}}}$ need
to be activated. This operator is expressed as 
\begin{equation}
\big(\forall (\boldsymbol{\mathsf{x}},\boldsymbol{\mathsf{y}})\in
\HHH\oplus\GGG\big)
\quad 
\boldsymbol{\mathsf{P}}_{\boldsymbol{\mathsf{\mathsf{V}}}}
\colon\boldsymbol(\boldsymbol{\mathsf{x}},\boldsymbol{\mathsf{y}})
\mapsto(\boldsymbol{\mathsf{t}},\boldsymbol{\mathsf{L}}
\boldsymbol{\mathsf{t}})
=(\boldsymbol{\mathsf{x}}-\boldsymbol{\mathsf{L}}^*\boldsymbol
{\mathsf{s}},\boldsymbol{\mathsf{y}}+\boldsymbol{\mathsf{s}})
\end{equation}
where $\boldsymbol{\mathsf{t}}=(\boldsymbol{\mathsf{Id}}+
\boldsymbol{\mathsf{L}}^*
\boldsymbol{\mathsf{L}})^{-1}(\boldsymbol{\mathsf{x}}+
\boldsymbol{\mathsf{L}}^*\boldsymbol{\mathsf{y}})$ and 
$\boldsymbol{\mathsf{s}}=
(\boldsymbol{\mathsf{Id}}+\boldsymbol{\mathsf{L}}
\boldsymbol{\mathsf{L}}^*)^{-1}
(\boldsymbol{\mathsf{L}}\boldsymbol{\mathsf{x}}-
\boldsymbol{\mathsf{y}})$ \cite[Lemma~3.1]{Optl14}.
This formula allows us to compute the components of 
$\boldsymbol{\mathsf{P}}_{\boldsymbol{\mathsf{\mathsf{V}}}}$,
which is especially simple when 
$\boldsymbol{\mathsf{Id}}+\boldsymbol{\mathsf{L}}^*
\boldsymbol{\mathsf{L}}$ or $\boldsymbol{\mathsf{Id}}+
\boldsymbol{\mathsf{L}}\boldsymbol{\mathsf{L}}^*$ is easily 
inverted.
\end{remark}

The previous result leads to a random block-coordinate primal-dual 
proximal algorithm for solving a wide range of structured convex 
optimization problems.

\begin{corollary}
Set $\mathsf{D}=\{0,1\}^{m+p}\smallsetminus
\{\boldsymbol{\mathsf{0}}\}$,
let $(\GG_k)_{1\leq k\leq p}$ be separable real Hilbert spaces, 
and set $\GGG=\GG_1\oplus\cdots\oplus\GG_p$. For every 
$i\in\{1,\ldots,m\}$, let $\mathsf{f}_i\in \Gamma_0(\HH_i)$ and, for 
every $k\in\{1,\ldots,p\}$, let $\mathsf{g}_k\in\Gamma_0(\GG_k)$, 
and let $\mathsf{L}_{ki}\colon\HH_i\to\GG_k$ be linear and bounded. 
It is assumed that there exists 
$(\mathsf{x}_1,\ldots,\mathsf{x}_m)\in\HHH$ such that
\begin{equation}
\label{egb18T4-29a}
(\forall i\in\{1,\ldots,m\})\quad 0\in
\partial\mathsf{f}_i(\mathsf{x}_i)+\sum_{k=1}^p\mathsf{L}_{ki}^*
\partial\mathsf{g}_k\bigg(\sum_{j=1}^m\mathsf{L}_{kj}
\mathsf{x}_j\bigg).
\end{equation}
Let $\boldsymbol{\mathsf{F}}$ be the set of solutions to the
problem
\begin{equation}
\minimize{\mathsf{x}_1\in\HH_1,\ldots,\mathsf{x}_m\in\HH_m}
{\sum_{i=1}^m\mathsf{f}_i(\mathsf{x}_i)+\sum_{k=1}^p
\mathsf{g}_k\bigg(\sum_{i=1}^m\mathsf{L}_{ki}\mathsf{x}_i\bigg)}
\end{equation}
and let $\boldsymbol{\mathsf{F}}^*$ be the set of 
solutions to the dual problem
\begin{equation}
\minimize{\mathsf{v}_1\in\GG_1,\ldots,\mathsf{v}_p\in\GG_p}
{\sum_{i=1}^m\mathsf{f}_i^*\bigg(-\Sum_{k=1}^p
\mathsf{L}_{ki}^*\mathsf{v}_k\bigg)+\sum_{k=1}^p
\mathsf{g}_k^*(\mathsf{v}_k)}.
\end{equation}
Let $\gamma\in\RPP$, let $(\mu_n)_{n\in\NN}$ 
be a sequence in $\left]0,2\right[$ such that 
$\inf_{n\in\NN}\mu_n>0$ and $\sup_{n\in\NN}\mu_n<2$,
let $\boldsymbol{x}_0$, $\boldsymbol{z}_0$,
$(\boldsymbol{a}_n)_{n\in\NN}$, and 
$(\boldsymbol{c}_n)_{n\in\NN}$ be $\HHH$-valued random variables,
let $\boldsymbol{y}_0$, $\boldsymbol{w}_0$,
$(\boldsymbol{b}_n)_{n\in\NN}$, and 
$(\boldsymbol{d}_n)_{n\in\NN}$ be $\GGG$-valued random 
variables, and let $(\boldsymbol{\varepsilon}_n)_{n\in\NN}$ be 
identically distributed $\mathsf{D}$-valued random variables. 
Define $\boldsymbol{\mathsf{V}}$ as in \eqref{egb18T4-10a} and 
set $\boldsymbol{\mathsf{P}}_{\boldsymbol{\mathsf{\mathsf{V}}}}
\colon\boldsymbol{\mathsf{x}}\mapsto
(\mathsf{Q}_j\boldsymbol{\mathsf{x}})_{1\leq j\leq m+p}$
where $(\forall i\in\{1,\ldots,m\})$ $\mathsf{Q}_i\colon
\HHH\oplus\GGG\to\HH_i$ and $(\forall k\in\{1,\ldots,p\})$ 
$\mathsf{Q}_{m+k}\colon\HHH\oplus\GGG\to\GG_k$, and iterate
\begin{equation}
\begin{array}{l}
\text{for}\;n=0,1,\ldots\\
\left\lfloor
\begin{array}{l}
\text{for}\;i=1,\ldots,m\\
\left\lfloor
\begin{array}{l}
z_{i,n+1}=z_{i,n}+\varepsilon_{i,n}\big(\mathsf{Q}_i(x_{1,n},%
\ldots,x_{m,n},y_{1,n},\ldots,y_{p,n})+c_{i,n}-z_{i,n}\big)\\[1mm]
x_{i,n+1}=x_{i,n}+\varepsilon_{i,n}\mu_n
\big(\prox_{\gamma\mathsf{f}_i}
(2z_{i,n+1}-x_{i,n})+a_{i,n}-z_{i,n+1}\big)
\end{array}
\right.\\
\text{for}\;k=1,\ldots,p\\
\left\lfloor
\begin{array}{l}
w_{k,n+1}=w_{k,n}+\varepsilon_{m+k,n}\big(\mathsf{Q}_{m+k}
(x_{1,n},\ldots,x_{m,n},y_{1,n},\ldots,y_{p,n})
+d_{k,n}-w_{k,n}\big)\\[1mm]
y_{k,n+1}=y_{k,n}+\varepsilon_{m+k,n}\mu_n
\big(\prox_{\gamma\mathsf{g}_k}
(2w_{k,n+1}-y_{k,n})+b_{k,n}-w_{k,n+1}\big).
\end{array}
\right.
\end{array}
\right.\\
\end{array}
\end{equation}
In addition, assume that conditions 
\ref{cgb18T4-09ii}--\ref{cgb18T4-09iiv} of
Corollary~\ref{cgb18T4-09} are satisfied.
Then $(\boldsymbol{z}_n)_{n\in\NN}$ converges weakly $\as$ to an 
$\boldsymbol{\mathsf{F}}$-valued random variable, and
$(\gamma^{-1}(\boldsymbol{w}_n-\boldsymbol{y}_n))_{n\in\NN}$ 
converges weakly $\as$ to an
$\boldsymbol{\mathsf{F}}^*$-valued random variable.
\end{corollary}
\begin{proof}
Using the same arguments as in \cite[Proposition~5.4]{Siop13} one
sees that this is an application of Corollary~\ref{cgb18T4-09} 
with, for every $i\in\{1,\ldots,m\}$, 
$\mathsf{A}_i=\partial\mathsf{f}_i$ and, for every 
$k\in\{1,\ldots,p\}$, $\mathsf{B}_k=\partial\mathsf{g}_k$.
\end{proof}

\begin{remark}
Sufficient conditions for \eqref{egb18T4-29a} to hold are provided
in \cite[Proposition~5.3]{Siop13}.
\end{remark}

\subsection{Random block-coordinate forward-backward splitting}
\label{sec:52}

The forward-backward algorithm addresses the problem of finding a 
zero of the sum of two maximally monotone operators, one of which 
has a strongly monotone inverse (see \cite{Sico10,Opti04} for
historical background). It has been applied to a wide variety of
problems among which mechanics, partial differential equations, 
best approximation, evolution inclusions, signal and image 
processing, convex optimization, learning theory, inverse 
problems, statistics, and game theory 
\cite{Sico10,Livre1,Bric13,Byrn14,Opti04,Banf11,Smms05,Devi11,%
Glow89,Merc79,Tsen90,Tsen91,Vill13}. 
In this section we design a block-coordinate version of this
algorithm with random sweeping and stochastic errors.

\begin{definition}{\rm\cite[Definition~2.3]{Sico10}}
\label{d:demir}
An operator $\mathsf{A}\colon\HH\to 2^{\HH}$ is \emph{demiregular} 
at $\mathsf{x}\in\dom\mathsf{A}$ if, for every sequence
$((\mathsf{x}_n,\mathsf{u}_n))_{n\in\NN}$ in 
$\gra\mathsf{A}$ and every $\mathsf{u}\in\mathsf{A}\mathsf{x}$ 
such that $\mathsf{x}_n\weakly\mathsf{x}$ and 
$\mathsf{u}_n\to\mathsf{u}$, we have $\mathsf{x}_n\to\mathsf{x}$.
\end{definition}

\begin{lemma}{\rm\cite[Proposition~2.4]{Sico10}}
\label{l:2009-09-20}
Let $\mathsf{A}\colon\HH\to 2^{\HH}$ be monotone and suppose that 
$\mathsf{x}\in\dom\mathsf{A}$. Then $\mathsf{A}$ is demiregular at
$\mathsf{x}$ in each of the 
following cases:
\begin{enumerate}
\item
\label{l:2009-09-20i}
$\mathsf{A}$ is uniformly monotone at $\mathsf{x}$, i.e., 
there exists an increasing function $\theta\colon\RP\to\RPX$ that 
vanishes only at $0$ such that
$(\forall\mathsf{u}\in\mathsf{A}\mathsf{x})
(\forall (\mathsf{y},\mathsf{v})\in\gra\mathsf{A})$
$\scal{\mathsf{x}-\mathsf{y}}{\mathsf{u}-\mathsf{v}}
\geq\theta(\|\mathsf{x}-\mathsf{y}\|)$.
\item
\label{l:2009-09-20ii}
$\mathsf{A}$ is strongly monotone, i.e., there exists 
$\alpha\in\RPP$ such that $\mathsf{A}-\alpha\Id$ is monotone.
\item
\label{l:2009-09-20iv-}
$\mathsf{J}_{\mathsf{A}}$ is compact, i.e., for every bounded set 
$\mathsf{C}\subset\HH$, the closure of 
$\mathsf{J}_{\mathsf{A}}(\mathsf{C})$ is compact. In particular, 
$\dom\mathsf{A}$ is boundedly relatively compact, i.e., the 
intersection of its closure with every closed ball is compact.
\item
\label{l:2009-09-20vi}
$\mathsf{A}\colon\HH\to\HH$ is single-valued with a single-valued
continuous inverse.
\item
\label{l:2009-09-20vii}
$\mathsf{A}$ is single-valued on $\dom\mathsf{A}$ and
$\Id-\mathsf{A}$ is demicompact.
\item
\label{p:2009-09-20ii+}
$\mathsf{A}=\partial\mathsf{f}$, where $\mathsf{f}\in\Gamma_0(\HH)$
is uniformly convex at $\mathsf{x}$, i.e., there exists an 
increasing function $\theta\colon\RP\to\RPX$ that vanishes only 
at $0$ such that 
\begin{equation}
(\forall\alpha\in\zeroun)(\forall\mathsf{y}\in\dom\mathsf{f})\quad
\mathsf{f}\big(\alpha \mathsf{x}+(1-\alpha)\mathsf{y}\big)
+\alpha(1-\alpha)\theta(\|\mathsf{x}-\mathsf{y}\|)
\leq\alpha\mathsf{f}(\mathsf{x})+(1-\alpha)\mathsf{f}(\mathsf{y}).
\end{equation}
\item
\label{p:2009-09-20ii++++}
$\mathsf{A}=\partial\mathsf{f}$, where $\mathsf{f}\in\Gamma_0(\HH)$ 
and, for every $\xi\in\RR$, 
$\menge{\mathsf{x}\in\HH}{\mathsf{f}(\mathsf{x})\leq\xi}$ is 
boundedly compact.
\end{enumerate}
\end{lemma}

Our block-coordinate forward-backward algorithm is the following.
\begin{proposition}
\label{pjk8yT4-03}
Set $\mathsf{D}=\{0,1\}^m\smallsetminus
\{\boldsymbol{\mathsf{0}}\}$ and,
for every $i\in\{1,\ldots,m\}$, let 
$\mathsf{A}_i\colon\HH_i\to 2^{\HH_i}$ be maximally monotone 
and let $\mathsf{B}_i\colon\HHH\to\HH_i$. Suppose that
\begin{equation}
\label{e:genna07-21}
(\exi\vartheta\in\RPP)(\forall\boldsymbol{\mathsf{x}}\in\HHH)
(\forall\boldsymbol{\mathsf{y}}\in\HHH)\quad
\sum_{i=1}^m\scal{\mathsf{x}_i-\mathsf{y}_i}{\mathsf{B}_i
\boldsymbol{\mathsf{x}}-\mathsf{B}_i\boldsymbol{\mathsf{y}}}
\geq\vartheta\sum_{i=1}^m\big\|\mathsf{B}_i\boldsymbol{\mathsf{x}}
-\mathsf{B}_i\boldsymbol{\mathsf{y}}\big\|^2,
\end{equation}
and that the set $\boldsymbol{\mathsf{F}}$ of solutions to the 
problem
\begin{equation}
\label{ejk8yT4-04x}
\text{find}\;\;{\mathsf{x}_1\in\HH_1,\ldots,\mathsf{x}_m\in\HH_m}
\;\;\text{such that}\;\;(\forall i\in\{1,\ldots,m\})\quad 0\in
\mathsf{A}_i\mathsf{x}_i+\mathsf{B}_i(\mathsf{x}_1,\ldots,
\mathsf{x}_m)
\end{equation}
is nonempty. Let $(\gamma_n)_{n\in\NN}$ be a sequence in 
$\left]0,2\vartheta\right[$ such that
$\inf_{n\in\NN}\gamma_n>0$ and $\sup_{n\in\NN}\gamma_n<2\vartheta$,
and let $(\lambda_n)_{n\in\NN}$ be a sequence in $\left]0,1\right]$
such that $\inf_{n\in\NN}\lambda_n>0$. 
Let $\boldsymbol{x}_0$, $(\boldsymbol{a}_n)_{n\in\NN}$, 
and $(\boldsymbol{c}_n)_{n\in\NN}$ be $\HHH$-valued random 
variables, and let $(\boldsymbol{\varepsilon}_n)_{n\in\NN}$ be 
identically distributed $\mathsf{D}$-valued random variables. 
Iterate
\begin{equation}
\label{e:main12}
\begin{array}{l}
\text{for}\;n=0,1,\ldots\\
\left\lfloor
\begin{array}{l}
\text{for}\;i=1,\ldots,m\\
\left\lfloor
\begin{array}{l}
x_{i,n+1}=x_{i,n}+\varepsilon_{i,n}\lambda_n
\big(\mathsf{J}_{\gamma_n\mathsf{A}_i}\big(x_{i,n}-\gamma_n
(\mathsf{B}_i(x_{1,n},\ldots,x_{m,n})+c_{i,n})\big)+
a_{i,n}-x_{i,n}\big),
\end{array}
\right.
\end{array}
\right.\\
\end{array}
\end{equation}
and set $(\forall n\in\NN)$
$\EEE_n=\sigma(\boldsymbol{\varepsilon}_n)$. Furthermore, 
assume that the following hold:
\begin{enumerate}
\item
\label{pjk8yT4-03ii}
$\sum_{n\in\NN}\sqrt{\EC{\|\boldsymbol{a}_n\|^2}{
\boldsymbol{\XX}_n}}<\pinf$ and
$\sum_{n\in\NN}\sqrt{\EC{\|\boldsymbol{c}_n\|^2}{
\boldsymbol{\XX}_n}}<\pinf$.
\item
\label{pjk8yT4-03iv-}
For every $n\in\NN$, $\EEE_n$ and $\XXX_n$ are independent.
\item
\label{pjk8yT4-03iv}
$(\forall i\in\{1,\ldots,m\})$ $\PP[\varepsilon_{i,0}=1]>0$.
\end{enumerate}
Then $(\boldsymbol{x}_n)_{n\in\NN}$ converges weakly $\as$ to an 
$\boldsymbol{\mathsf F}$-valued random variable $\boldsymbol{x}$. 
If, in addition, one of the following holds:
\begin{enumerate}
\setcounter{enumi}{3}
\item
\label{pjk8yT4-03v}
for every $\boldsymbol{\mathsf{x}}\in\boldsymbol{\mathsf{F}}$ and
every $i\in\{1,\ldots,m\}$, $\mathsf{A}_i$ is demiregular at 
$\mathsf{x}_i$;
\item
\label{pjk8yT4-03vi}
the operator $\boldsymbol{\mathsf{x}}
\mapsto(\mathsf{B}_i\boldsymbol{\mathsf{x}})_{1\leq i\leq m}$ is
demiregular at every point in $\boldsymbol{\mathsf{F}}$;
\end{enumerate}
then $(\boldsymbol{x}_n)_{n\in\NN}$ converges strongly $\as$ to 
$\boldsymbol{x}$.
\end{proposition}
\begin{proof}
We are going to apply Theorem~\ref{t:1}.
Set $\boldsymbol{\mathsf A}\colon\HHH\to 2^\HHH\colon
\boldsymbol{\mathsf x}\mapsto\cart_{\!i=1}^{\!m}\mathsf{A}_i
\mathsf{x}_i$, 
$\boldsymbol{\mathsf B}\colon\HHH\to\HHH\colon\boldsymbol{\mathsf x}
\mapsto(\mathsf{B}_i\boldsymbol{\mathsf x})_{1 \leq i\leq m}$,
and, for every $n\in\NN$, $\alpha_n=1/2$,
$\beta_n=\gamma_n/(2\vartheta)$,
$\boldsymbol{\mathsf T}_{\!n}=\boldsymbol{\mathsf J}_{\gamma_n
\boldsymbol{\mathsf A}}$, 
$\boldsymbol{\mathsf R}_n=\ID-\gamma_n\boldsymbol{\mathsf{B}}$,
and $\boldsymbol{b}_n=-\gamma_n\boldsymbol{c}_n$. Then, 
$\boldsymbol{\mathsf F}=\zer(\boldsymbol{\mathsf A}+
\boldsymbol{\mathsf B})$ and, for every 
$n\in\NN$, $\boldsymbol{\mathsf T}_{\!n}$ is $\alpha_n$-averaged 
\cite[Corollary~23.8]{Livre1},
$\boldsymbol{\mathsf T}_{\!n}\colon\boldsymbol{\mathsf{x}}\mapsto
(\mathsf{J}_{\gamma_n\mathsf{A}_i}\mathsf{x}_i)_{1\leq i\leq m}$ 
\cite[Proposition~23.16]{Livre1},
$\boldsymbol{\mathsf R}_n$ is $\beta_n$-averaged
\cite[Proposition~4.33]{Livre1}, and
$\Fix(\boldsymbol{\mathsf T}_{\!n}\circ\boldsymbol{\mathsf R}_n)=
\boldsymbol{\mathsf F}$ \cite[Proposition~25.1(iv)]{Livre1}. 
Moreover, $\sum_{n\in\NN}\sqrt{\EC{\|\boldsymbol{b}_n\|^2}{
\boldsymbol{\XX}_n}}\leq2\vartheta\sum_{n\in\NN}
\sqrt{\EC{\|\boldsymbol{c}_n\|^2}{\boldsymbol{\XX}_n}}<\pinf$ and 
\eqref{e:main12} is a special case of \eqref{e:main1}. 
Observe that \eqref{egb18T4-23a} and \eqref{egb18T4-23b}
imply the existence of $\widetilde{\Omega}\in\FF$ such that 
$\PP(\widetilde{\Omega})=1$ and
\begin{equation}
\label{eQ6t5Ds2-31b}
(\forall\omega\in\widetilde{\Omega})
(\forall\boldsymbol{\mathsf{z}}\in\boldsymbol{\mathsf{F}})\quad
\begin{cases}
\boldsymbol{\mathsf T}_{\!n}(\boldsymbol{\mathsf{R}}_n
\boldsymbol{x}_n(\omega))
-\boldsymbol{\mathsf{R}}_n\boldsymbol{x}_n(\omega)
+\boldsymbol{\mathsf R}_n\boldsymbol{\mathsf{z}}
\to\boldsymbol{\mathsf{z}}\\
\boldsymbol{\mathsf{R}}_n\boldsymbol{x}_n(\omega)-
\boldsymbol{x}_n(\omega)-\boldsymbol{\mathsf R}_n
\boldsymbol{\mathsf{z}}\to-\boldsymbol{\mathsf{z}}.
\end{cases}
\end{equation}
Consequently, since $\inf_{n\in\NN}\gamma_n>0$, 
\begin{equation}
\label{eQ6t5Ds2-31c}
(\forall\omega\in\widetilde{\Omega})
(\forall\boldsymbol{\mathsf{z}}\in\boldsymbol{\mathsf{F}})\quad
\begin{cases}
\boldsymbol{\mathsf J}_{\gamma_n\boldsymbol{\mathsf A}}
\big(\boldsymbol{x}_n(\omega)
-\gamma_n\boldsymbol{\mathsf{B}}\boldsymbol{x}_n(\omega)\big)-
\boldsymbol{x}_n(\omega)=
\boldsymbol{\mathsf T}_{\!n}\big(\boldsymbol{\mathsf{R}}_n
\boldsymbol{x}_n(\omega)\big)-\boldsymbol{x}_n(\omega)
\to\boldsymbol{0}\\
\boldsymbol{\mathsf{B}}\boldsymbol{x}_n(\omega)\to
\boldsymbol{\mathsf{B}}\boldsymbol{\mathsf{z}}.
\end{cases}
\end{equation}
Now set
\begin{equation}
\label{e:burn}
(\forall n\in\NN)\quad\boldsymbol{y}_n
=\boldsymbol{\mathsf{J}}_{\gamma_n\boldsymbol{\mathsf{A}}}
(\boldsymbol{x}_n-\gamma_n
\boldsymbol{\mathsf{B}}\boldsymbol{x}_n)
\quad\text{and}\quad
\boldsymbol{{u}}_n=\gamma_n^{-1}(\boldsymbol{{x}}_n
-\boldsymbol{{y}}_n)-
\boldsymbol{\mathsf{B}}\boldsymbol{{x}}_n.
\end{equation}
Then \eqref{eQ6t5Ds2-31c} yields
\begin{equation}
\label{eQ6t5Ds2-31d}
(\forall\omega\in\widetilde{\Omega})
(\forall\boldsymbol{\mathsf{z}}\in\boldsymbol{\mathsf{F}})\quad
\boldsymbol{x}_n(\omega)-\boldsymbol{y}_n(\omega)\to
\boldsymbol{\mathsf{0}}
\quad\text{and}\quad \boldsymbol{u}_n(\omega)\to
-\boldsymbol{\mathsf{B}}\boldsymbol{\mathsf{z}}.
\end{equation}
Now, let us establish condition \ref{t:1iii} of Theorem~\ref{t:1}.
To this end, it is enough to fix 
$\boldsymbol{\mathsf{z}}\in\boldsymbol{\mathsf{F}}$,
$\boldsymbol{\mathsf{x}}\in\HHH$, a strictly 
increasing sequence $(k_n)_{n\in\NN}$ in $\NN$, and 
$\omega\in\widetilde{\Omega}$ such that
$\boldsymbol{x}_{k_n}(\omega)\weakly\boldsymbol{\mathsf{x}}$
and to show that 
$\boldsymbol{\mathsf{x}}\in\boldsymbol{\mathsf{F}}$. 
It follows from \eqref{eQ6t5Ds2-31c} that
$\boldsymbol{\mathsf{B}}\boldsymbol{x}_{k_n}(\omega)\to
\boldsymbol{\mathsf{B}}\boldsymbol{\mathsf{z}}$. Hence, 
since \cite[Example~20.28]{Livre1} asserts that 
$\boldsymbol{\mathsf{B}}$ is maximally monotone, we deduce from
\cite[Proposition~20.33(ii)]{Livre1} that 
$\boldsymbol{\mathsf{B}}\boldsymbol{\mathsf{x}}=
\boldsymbol{\mathsf{B}}\boldsymbol{\mathsf{z}}$. We also derive
from \eqref{eQ6t5Ds2-31d} that 
$\boldsymbol{y}_{k_n}(\omega)\weakly\boldsymbol{\mathsf{x}}$ and
$\boldsymbol{u}_{k_n}(\omega)\to-\boldsymbol{\mathsf{B}}
\boldsymbol{\mathsf{z}}=-\boldsymbol{\mathsf{B}}
\boldsymbol{\mathsf{x}}$. Since \eqref{e:burn} implies that 
$(\boldsymbol{y}_{k_n}(\omega),%
\boldsymbol{u}_{k_n}(\omega))_{n\in\NN}$ 
lies in the graph of $\boldsymbol{\mathsf{A}}$, it follows from 
\cite[Proposition~20.33(ii)]{Livre1} that $-\boldsymbol{\mathsf{B}}
\boldsymbol{\mathsf{x}}\in\boldsymbol{\mathsf{A}}\boldsymbol
{\mathsf{x}}$, i.e., 
$\boldsymbol{\mathsf{x}}\in\boldsymbol{\mathsf{F}}$.
This proves that $(\boldsymbol{x}_n)_{n\in\NN}$ converges weakly 
$\as$ to an $\boldsymbol{\mathsf F}$-valued random variable 
$\boldsymbol{x}$, say 
\begin{equation}
\label{eQ6t5Ds2-31a}
\boldsymbol{x}_n(\omega)\weakly\boldsymbol{x}(\omega)
\end{equation}
for every $\omega$ in some $\widehat{\Omega}\in\FF$ such that
$\widehat{\Omega}\subset\widetilde{\Omega}$ and 
$\PP(\widehat{\Omega})=1$.

Finally take $\omega\in\widehat\Omega$.
First, suppose that \ref{pjk8yT4-03v} holds. Then
$\boldsymbol{\mathsf{A}}$ is demiregular at 
$\boldsymbol{x}(\omega)$. In view of \eqref{eQ6t5Ds2-31d}
and \eqref{eQ6t5Ds2-31a},
$\boldsymbol{y}_n(\omega)\weakly\boldsymbol{x}(\omega)$. 
Furthermore,
$\boldsymbol{u}_n(\omega)\to-\boldsymbol{\mathsf{B}}
\boldsymbol{x}(\omega)$ and 
$(\boldsymbol{y}_n(\omega),%
\boldsymbol{u}_{n}(\omega))_{n\in\NN}$ 
lies in the graph of $\boldsymbol{\mathsf{A}}$. Altogether
$\boldsymbol{y}_n(\omega)\to\boldsymbol{x}(\omega)$ and,
therefore $\boldsymbol{x}_n(\omega)\to\boldsymbol{x}(\omega)$.
Now, suppose that \ref{pjk8yT4-03vi} holds. Then,
since \eqref{eQ6t5Ds2-31c} yields 
$\boldsymbol{\mathsf{B}}\boldsymbol{x}_n(\omega)\to
\boldsymbol{\mathsf{B}}\boldsymbol{x}(\omega)$, 
\eqref{eQ6t5Ds2-31a} implies that 
$\boldsymbol{x}_n(\omega)\to\boldsymbol{x}(\omega)$.
\end{proof}

\begin{remark}
\label{rjk8yT4-04}
Here are a few remarks regarding Proposition~\ref{pjk8yT4-03}.
\begin{enumerate}
\item
Proposition~\ref{pjk8yT4-03} generalizes 
\cite[Corollary~6.5 and Remark~6.6]{Opti04}, which does not 
allow for block-processing and uses deterministic variables. 
\item
Problem~\eqref{ejk8yT4-04x} was considered in \cite{Sico10}, 
where it was shown to capture formulations encountered in areas 
such as evolution equations, game theory, optimization, best 
approximation, and network flows. It also models domain 
decomposition problems in partial differential equations 
\cite{Luis11}.
\item
Proposition~\ref{pjk8yT4-03} generalizes 
\cite[Theorem~2.9]{Sico10}, which uses a fully parallel deterministic
algorithm in which all the blocks are used at each iteration, i.e., 
$(\forall n\in\NN)(\forall i\in\{1,\ldots,m\})$ 
$\varepsilon_{i,n}=1$.
\item
As shown in \cite{Svva10,Opti12}, strongly monotone composite
inclusion problems can be solved by applying the forward-backward 
algorithm to the dual problem. Using 
Proposition~\ref{pjk8yT4-03} we can obtain a block-coordinate 
version of this primal-dual framework. Likewise, it was shown in 
\cite{Opti12,Cond13,Bang13} that suitably renormed versions
of the forward-backward algorithm applied in the primal-dual space
yielded a variety of methods for solving composite inclusions in
duality. Block-coordinate versions of these methods can be devised
via Proposition~\ref{pjk8yT4-03}.
\end{enumerate}
\end{remark}

Next, we present an application of Proposition~\ref{pjk8yT4-03}
to block-coordinate convex minimization.

\begin{corollary}
\label{cjk8yT4-04}
Set $\mathsf{D}=\{0,1\}^m\smallsetminus\{\boldsymbol{\mathsf{0}}\}$ 
and let $(\GG_k)_{1\leq k\leq p}$ be separable real Hilbert spaces. 
For every $i\in\{1,\ldots,m\}$, let 
$\mathsf{f}_i\in\Gamma_0(\HH_i)$ and, for every 
$k\in\{1,\ldots,p\}$, let $\tau_k\in\RPP$, let 
$\mathsf{g}_k\colon\GG_k\to\RR$ be a differentiable convex function 
with a $\tau_k$-Lipschitz-continuous gradient, and let 
$\mathsf{L}_{ki}\colon\HH_i\to\GG_k$ be linear and bounded. It is 
assumed that 
$\min_{1\leq k\leq p}\sum_{i=1}^m\|\mathsf{L}_{ki}\|^2>0$ and that
the set $\boldsymbol{\mathsf{F}}$ of solutions to the problem
\begin{equation}
\label{ejk8yT4-04m}
\minimize{\mathsf{x}_1\in\HH_1,\ldots,\mathsf{x}_m\in\HH_m}
{\sum_{i=1}^m\mathsf{f}_i(\mathsf{x}_i)+\sum_{k=1}^p
\mathsf{g}_k\bigg(\sum_{i=1}^m\mathsf{L}_{ki}\mathsf{x}_i\bigg)}
\end{equation}
is nonempty. Let
\begin{equation}
\label{egb18T4-06}
\vartheta\in\bigg]0,\bigg(\displaystyle{\sum_{k=1}^p}
\:\tau_k\bigg\|\sum_{i=1}^m\mathsf{L}_{ki}\mathsf{L}_{ki}^*
\bigg\|\bigg)^{-1}\:\bigg],
\end{equation}
let $(\gamma_n)_{n\in\NN}$ be a sequence in 
$\left]0,2\vartheta\right[$ such that
$\inf_{n\in\NN}\gamma_n>0$ and $\sup_{n\in\NN}\gamma_n<2\vartheta$,
and let $(\lambda_n)_{n\in\NN}$ be a sequence in $\left]0,1\right]$
such that $\inf_{n\in\NN}\lambda_n>0$. 
Let $\boldsymbol{x}_0$, $(\boldsymbol{a}_n)_{n\in\NN}$, 
and $(\boldsymbol{c}_n)_{n\in\NN}$ be $\HHH$-valued random 
variables, and let $(\boldsymbol{\varepsilon}_n)_{n\in\NN}$ be 
identically distributed $\mathsf{D}$-valued random variables. 
Iterate
\begin{equation}
\label{e:main14}
\begin{array}{l}
\text{for}\;n=0,1,\ldots\\
\left\lfloor
\begin{array}{l}
\text{for}\;i=1,\ldots,m\\
\left\lfloor
\begin{array}{l}
r_{i,n}=\varepsilon_{i,n}\big(x_{i,n}-\gamma_n\big(\sum_{k=1}^p
\mathsf{L}_{ki}^*\nabla\mathsf{g}_k\big(\sum_{j=1}^m\mathsf{L}_{kj}
x_{j,n}\big)+c_{i,n}\big)\big)\\[2mm]
x_{i,n+1}=x_{i,n}+\varepsilon_{i,n}\lambda_n\big(\prox_{\gamma_n
\mathsf{f}_i}r_{i,n}+a_{i,n}-x_{i,n}\big).
\end{array}
\right.
\end{array}
\right.\\
\end{array}
\end{equation}
In addition, assume that conditions 
\ref{pjk8yT4-03ii}--\ref{pjk8yT4-03iv} in 
Proposition~\ref{pjk8yT4-03} are satisfied.
Then $(\boldsymbol{x}_n)_{n\in\NN}$ converges weakly $\as$ to an 
$\boldsymbol{\mathsf{F}}$-valued random variable.
If, furthermore, one of the following holds (see 
Lemma~\ref{l:2009-09-20}\ref{p:2009-09-20ii+}%
--\ref{p:2009-09-20ii++++} for examples):
\begin{enumerate}
\item
\label{pjk8yT4-03V}
for every $\boldsymbol{\mathsf{x}}\in\boldsymbol{\mathsf{F}}$ and
every $i\in\{1,\ldots,m\}$, $\partial\mathsf{f}_i$ is demiregular at 
$\mathsf{x}_i$;
\item
\label{pjk8yT4-03VI}
the operator $\boldsymbol{\mathsf{x}}\mapsto(\sum_{k=1}^p
\mathsf{L}_{ki}^*\nabla\mathsf{g}_k(
\sum_{j=1}^m\mathsf{L}_{kj}\mathsf{x}_j))_{1\leq i\leq m}$ 
is demiregular at every point in $\boldsymbol{\mathsf{F}}$;
\end{enumerate}
then $(\boldsymbol{x}_n)_{n\in\NN}$ converges strongly $\as$ to 
$\boldsymbol{x}$.
\end{corollary}
\begin{proof}
As shown in \cite[Section~4]{Sico10}, \eqref{ejk8yT4-04m} is 
a special case of \eqref{ejk8yT4-04x} with
\begin{equation}
\label{e:genna07-5}
\mathsf{A}_i=\partial\mathsf{f}_i\quad\text{and}\quad
\mathsf{B}_i\colon(\mathsf{x}_j)_{1\leq j\leq m}\mapsto\sum_{k=1}^p
\mathsf{L}_{ki}^*\nabla\mathsf{g}_k\bigg(\sum_{j=1}^m\mathsf{L}_{kj}
\mathsf{x}_j\bigg).
\end{equation}
Now set $\boldsymbol{\mathsf h}\colon\HHH\to\RR\colon
\boldsymbol{\mathsf x}\mapsto\sum_{k=1}^p\mathsf{g}_k
\big(\sum_{i=1}^m\mathsf{L}_{ki}\mathsf{x}_i\big)$. Then 
$\boldsymbol{\mathsf h}$ is a Fr\'echet-differentiable convex
function and $\boldsymbol{\mathsf B}=\nabla\boldsymbol{\mathsf h}$ 
is Lipschitz-continuous with constant $1/\vartheta$, where 
$\vartheta$ is given in \eqref{egb18T4-06}. It therefore follows 
from the Baillon-Haddad theorem \cite[Theorem~18.15]{Livre1} that
\eqref{e:genna07-21} holds with this constant. Since, in view of 
\eqref{e:prox2}, \eqref{e:main12} specializes to \eqref{e:main14}, 
the convergence claims follow from 
Proposition~\ref{pjk8yT4-03}.
\end{proof}

\begin{remark}
\label{rjk8yT4-05}
Here are a few observations about Corollary~\ref{cjk8yT4-04}.
\begin{enumerate}
\item
If more assumptions are available about the problem, the Lipschitz
constant $\vartheta$ of \eqref{egb18T4-06} can be improved. Some 
examples are given in \cite{Nmtm09}.
\item
Recently, some block-coordinate forward-backward methods have been 
proposed for not necessarily convex minimization problems in
Euclidean spaces. Thus, when applied to convex functions satisfying 
the Kurdyka-{\L}ojasiewicz inequality, the deterministic 
block-coordinate forward-backward algorithm proposed in 
\cite[Section~3.6]{Bolt14} corresponds to the special case of 
\eqref{ejk8yT4-04m} in which 
\begin{equation}
\label{egb18T4-07b}
\HHH\;\text{is a Euclidean space},\quad
p=1,\quad\text{and}\quad
(\forall\boldsymbol{\mathsf{x}}\in\HHH)\quad
\sum_{i=1}^m\mathsf{L}_{1i}\mathsf{x}_i=\boldsymbol{\mathsf{x}}.
\end{equation}
In that method, the sweeping proceeds by activating 
only one block at each iteration according to a periodic schedule. 
Moreover, errors and relaxations are not allowed. This approach was 
extended in \cite{Chou14} to an error-tolerant form with a 
cyclic sweeping rule whereby each block is used at least once
within a preset number of consecutive iterations.
\item
A block-coordinate forward-backward method with random seeping was
proposed in \cite{Rich14} in the special case of 
\eqref{egb18T4-07b}. That method uses only one block at each 
iteration, no relaxation, and no error terms. The asymptotic 
analysis of \cite{Rich14} provides a lower bound on the 
probability that $(f+g_1)(x_n)$ be close to $\inf(f+g_1)(\HH)$, 
with no result on the convergence of the sequence 
$(x_n)_{n\in\NN}$. Related work is presented in 
\cite{Luxi15,Neco14}.
\end{enumerate}
\end{remark}

\end{document}